\documentclass[11pt, amscd, psamsfonts]{amsart}

\usepackage{amsthm,   amsmath}
\usepackage{amssymb} 
\usepackage{amscd}
\usepackage{enumerate}

\usepackage[all]{xy}
\usepackage{verbatim}

\usepackage{graphicx}
\usepackage{amsfonts}
\usepackage{mathrsfs}
\usepackage{color}
\usepackage{soul}
\usepackage{hyperref}

  \usepackage[margin=1.40in]{geometry}

\usepackage{amsthm}

\newtheorem{theorem}{Theorem}[subsection]
\newtheorem{question}[theorem]{Question}

\newtheorem{cor}[theorem]{Corollary}

\newtheorem{lemma}[theorem]{Lemma}
\newtheorem{proposition}[theorem]{Proposition}
\newtheorem{subtheorem}{Theorem}[subsubsection]
\newtheorem{subcorollary}[subtheorem]{Corollary}

\theoremstyle{remark}
\newtheorem{remark}[theorem]{Remark}

\newtheorem{example}[theorem]{Example}

\theoremstyle{definition}
\newtheorem{definition}[theorem]{Definition}

\newcommand{\quash}[1]{}

\usepackage{marginnote}

\DeclareFontEncoding{OT2}{}{}

\newcommand{\mb}[1]{\mathbb{#1}}

\newcommand{\N}{\mb{N}}

\newcommand{\Q}{\mb{Q}}

\newcommand{\C}{\mb{C}}
\newcommand{\R}{\mb{R}}

\def\R{\mathbb{R}}

\def\Q{\mathbb{Q}}

\def\be{\begin{equation}}
\def\ee{\end{equation}}

\def\arrowdown#1#2{\Big\downarrow \rlap{$\vcenter{\hbox{$\scriptstyle#2$}}$}
{\hbox to -10pt{\hss{$\vcenter{\hbox{$\scriptstyle#1$}}$}}}}
\def\arrowup#1#2{\Big\uparrow \rlap{$\vcenter{\hbox{$\scriptstyle#2$}}$}
{\hbox to -10pt{\hss{$\vcenter{\hbox{$\scriptstyle#1$}}$}}}}

 \numberwithin{equation}{section}

\begin{document}

\title{On the slopes of the lattice of sections of hermitian line bundles}

\author[Chinburg]{T. Chinburg}
\address{T. Chinburg, Dept. of Mathematics\\ Univ. of Pennsylvania \\ Philadelphia, PA 19104, USA}
\email{ted@math.upenn.edu}
\thanks{T. C. was partially supported by  NSF FRG Grant No.\ DMS-1360767, NSF FRG Grant No.\ DMS-1265290,
NSF SaTC grant No. CNS-1513671  and Simons Foundation Grant No.\ 338379.}

\author[Guignard]{Q. Guignard}
\address{Q. Guignard\\I.H.E.S.\\
35, Route de Chartres\\
F-91440 Bures-sur-Yvette, France}
\email{guignard@ihes.fr}
\thanks{Q. G. was partially supported by the \'Ecole Normale Sup\'erieure and the I.H.E.S.}

\author[Soul\'e]{C. Soul\'e}
\address{C. Soul\'e\\I.H.E.S.\\
35, Route de Chartres\\
F-91440 Bures-sur-Yvette, France}
\email{soule@ihes.fr}
\thanks{C.S. was partially supported by the C.N.R.S..}


 \date{\today}

\maketitle

\begin{abstract}In this paper we apply Arakelov theory to study the distribution of the Petersson norms  of classical cusp forms as well
as the distribution of  the sup norms of rational functions on adelic subsets of curves.  The method in both cases is to study the limiting distribution of the 
successive minima of norms of global sections of powers of a metrized ample line bundle as one takes increasing powers of the bundle.  We develop a general method for computing the measure associated to this distribution.  We also study measures associated to the  zeros of sections which have small norm.  
\end{abstract}

\tableofcontents

\section{Introduction}
\label{s:intro}
The development of Arakelov theory has benefited from a close study of
applications to classical questions.  The proofs of the conjectures of Mordell and Lang are 
famous examples.  We study in this paper  the distribution of norms of two kinds of
classical objects.   The first consists of the Petersson norms of modular forms with integral Fourier coefficients and increasing weight for $\mathrm{SL}_2(\mathbb{Z})$.  The second consists of the distribution of sup norms of 
polynomials with integer coefficients on compact subsets of the complex plane.  More generally, we consider  the
sup norms of rational functions with prescribed poles on adelic subsets of curves over number fields. These subjects are linked by the fact that they both concern the successive minima of the norms of global sections of increasing powers of metrized line bundles on arithmetic surfaces.    We treat both subjects in this paper because there is a substantial overlap in the underlying theory needed to study them.  

Finding successive minima of norms of global sections of powers of metrized line bundles has a long history  in Arakelov theory.  The arithmetic Hilbert-Samuel theorem (\cite{GS-Comptes}, \cite{AbbesBouche}) concerns the existence of sections with small norm.  In \cite{Chen}, Chen 
developed a theory of convergence for distributions associated to the successive minima of sequences of lattices.  He applied this theory
to show the existence of limiting distributions associated to the successive minima of norms of sections of increasing powers of line bundles
with smooth metrics on arithmetic varieties.  
For our applications  we need to work with some
particular metrics which are not smooth, using work on such metrics developed by Bost \cite{BostLetter} and K\"uhn \cite{Kuhn}.
In the case of Petersson norms of cusp forms, this leads to a new phenomenon not appearing in the work of Chen.  Namely, the limiting distribution associated to the successive minima of norms as the weight of the cusp forms increases does not have compact support.  

One consequence of our results has to do with congruences between modular forms.
We show that most of the small successive minima of the Petersson norms of cusp forms with integral Fourier coefficients arise from non-trivial congruences between Hecke eigenforms.  To see why congruences lead to small Petersson norms, 
suppose $f_1$ and $f_2$ are distinct normalized Hecke eigencuspforms, so that the first coefficient in each of their Fourier expansions at infinity
is $1$.  A non-trivial congruence  between these forms amounts to the statement that $g = (f_1 - f_2)/m$ has integral Fourier coefficients for
some integer $m > 1$.  In this case, $g$ will often have smaller Petersson norm than either $f_1$ or $f_2$.  More general congruences
involving several eigenforms are involved in the precise statements of our results in Definition \ref{def:nocongruence} and Theorem \ref{thm:succ}(iii).

Classical arithmetic capacity theory was motivated by 
the problem of finding whether there is a non-zero polynomial with integer coefficients which has sup norm less than one on a given subset of the complex plane.   The generalization of this problem to arbitrary curves involves studying global sections of powers of lines bundles which have particular Green's metrics.  Classical capacity theory produces an upper bound for the minimal such sup norm
which is not sharp in general.  We develop in this paper an approach via local Chebyshev constants  for obtaining better bounds over schemes of arbitrary dimension, and we obtain additional information on successive minima.
This leads to new results about classical questions. 

For instance, suppose $E$ is a compact subset of the complex plane which is
invariant under complex conjugation.  Let $m(n,E)$ be the minimal sup norm over $E$ of a non-zero polynomial with integer coefficients and 
degree $n$.  Since $m(\ell + n,E) \le m(\ell,E) \cdot m(n,E)$, the classical Fekete Lemma \cite[p. 10]{Chen} shows
$M(E) = \lim_{n \to \infty} m(n,E)^{1/n}$ exists.  Classical capacity theory as in \cite{Rumely1, Rumely2} shows that if the capacity $\gamma(E)$
of $E$ satisfies $\gamma(E) < 1$ then $M(E) < 1$.   Our work on local Chebyshev constants provides more precise information about $M(E)$.
As an example, suppose $E$ is the closed disk of radius $1/2$ centered at $1/2$. Then $\gamma(E) = 1/2$, and we will use the
machinery of \S \ref{s:Chebyshev} to show $0.64 < M(E) < 0.67$ (see Example \ref{ex:quarterdisc}).

The Chebyshev method is useful for showing that in some cases, the successive minima are almost all equal.  In this case, one says  the associated metrized bundles are asymptotically semi-stable,
and the limiting measure associated to successive minima is the Dirac measure supported on $0$. We will show that this situation arises from   adelic subsets of curves which have capacity one.  Motivated by work of Serre on the distribution of eigenvalues of Frobenius on abelian varieties, we will also study the distribution of zeros of sections of small norm with respect to
capacity theoretic metrics.  We will show that in the case of adelic sets of capacity one, one can find sections
of approximately minimal norm whose zeros tend toward the associated equilibrium distribution while avoiding any prescribed finite set of points.

A careful reader will notice that the classical questions we study involving Petersson norms of cusp
forms and the capacities of adelic sets lead to considering particular metrics on line bundles.
While some of our results could be generalized to other metrics, we prefer to focus
on the cases at hand.  Similarly, we focus on the Petersson norms of cusp forms for $\mathrm{SL}_2(\mathbb{Z})$ rather than on developing in this paper generalizations to arbitrary modular forms on reductive groups.  Such generalizations are naturally of interest.  However, in this paper we are concerned with demonstrating the possibility of obtaining explicit results.  For example, we will show that the limiting measure associated to Petersson inner products
of cusp forms for $\mathrm{SL}_2(\mathbb{Z})$ has support bounded above by $2 \pi + 6(1 - \mathrm{log}(12)) = -2.62625...$.  We hope a detailed analysis of
the $\mathrm{SL}_2(\mathbb{Z})$ case will motivate future research on more general modular forms.

\smallskip
This paper is organized in the following way. 

\smallskip
In \S \ref{s:basicobjects} we begin by recalling various kind of slopes associated to an hermitian adelic vector bundle over a number field.  The
example of primary interest is provided by the global sections of an ample metrized line bundle on an arithmetic variety.
The naive adelic slopes associated to such sections $s$ arise from a height $\lambda(s)$ recalled in Definition \ref{def:slopes}.  Here $\lambda(s)$  is
the negative of the natural logarithmic norm of $s$.  For this reason, the successive minima of norms of sections correspond to successive
maxima of heights.  We recall in \S \ref{s:basicobjects} 
 some results of Chen \cite{Chen} concerning various kinds of successive maxima of heights associated to the global sections of metrized line bundles.

In \S \ref{s:modular} we consider slopes associated to lattices of cusp forms $f$ of increasing weight for $\mathrm{SL}_2(\mathbb{Z})$
which have integral $q$-expansions.  
We begin by recalling work of K\"uhn and Bost concerning the interpretation of Petersson norms of such cusp forms via Arakelov theory.
When the g.c.d. of the Fourier coefficients is one, the height  $\lambda(f)$ of $f$ is simply one half 
the negative of the logarithm of the Petersson norm of $f$.  
A key issue is that the adelic metrics which arise on the line bundle $L$ appropriate to this application are singular at infinity. Thus
one cannot apply Chen's work directly.  Instead
we consider forms which vanish to at least prescribed orders at infinity, and then let these orders tend to $0$. An interesting
conclusion in our main result, Theorem \ref{thm:succ}, is that the the probability measure $\nu$ which results in limit of large weights
has support bounded above but not bounded below.  In Definition \ref{def:nocongruence} we define a nonzero cusp form $f $ to not arise from a congruence between Hecke eigenforms if when we write $f$ as a linear combination  $\sum_i c_i f_i$ of distinct normalized eigenforms $f_i$, the $c_i$ are algebraic integers divisible
in the ring of all algebraic integers by the g.c.d. of the Fourier coefficients of $f$.  We will show that  Petersson norms of such $f$ are very large and contribute a vanishingly small proportion of successive minima as the weight tends to infinity.
The measure $\nu$ thus has to do with non-trivial congruences between eigenforms which give rise to forms with integral
$q$ expansions having much smaller Petersson norms.  

In \S \ref{s:ArakelovSection}, we consider smooth projective curves $X$ of positive genus.  Building on work of one of us in \cite{MI}, we 
deduce an explicit upper bound on the largest minimum of $H^0 (X , L^{\otimes n})$  in terms of arithmetic intersection numbers. However, this result falls short of proving the asymptotic semi-stability of the metrics on $H^0 (X , L^{\otimes n})$.

In \S \ref{s:Chebyshev} we will apply the theory of Okunkov bodies to study successive maxima of heights  for $X$ of any dimension.  We introduce
local and global Chebyshev transforms which are maps from the  Okounkov body of $X$ to the real numbers. The global Chebyshev transform is the sum of the local ones. We prove in Corollary \ref{s:diracmeasureChebychev} that, if the global Chebyshev transform is a constant function, the limit distribution ${\nu}$ is a Dirac measure.   We compute explicitly the local Chebyshev transforms in some particular cases
when $X$ is a projective space.  The main strength
 of this technique is that in some cases one can compute explicitly the limit distributions
 of the successive maxima associated to heights.  

In \S \ref{s:SerreMeasures} we study the distribution of zeros of those sections of powers of a metrized line bundle 
which have at least a prescribed height, i.e. those sections whose norms are small in the corresponding way.
We begin with an example in \S \ref{s:AnExample} which suggests that  sections of ``small" norm may have to have at least some of their
zeros at particular points, the remaining zeros being variable.   To formulate this precisely we recall a result of
Serre concerning the decomposition in to atomic and diffuse parts of limits of measures in the weak topology
on the space of positive Radon measures.  The connection of this theory to zeros of cusp forms of small Petersson norms
is discussed in Remark \ref{rem:HS} and Question \ref{q:askit}.  

In \S \ref{s:Capacityone} we consider applications to adelic capacity theory. This has to do
with the possible sup norms of rational functions on adelic subsets of curves.  
  We will
 apply work of Rumely to show that in the case of capacity metrics associated to adelic
 sets of capacity one, the associated metrized bundles are asymptotically semi-stable,
and the measure $\nu$ is the Dirac measure supported at $0$.   We will
 also study the locations of the zeros of sections which arise in this case using the work
 in \S \ref{s:SerreMeasures}.

\smallskip

\subsection*{Acknowledgements}
T.\,C. would like to thank  the I.H.E.S. for support during 
the Fall of 2015. Q.G. took part in this project during the preparation of his Ph.D. thesis, and would like to thank I.H.E.S. and E.N.S for hospitality and support in that period.

 \section{Semistability, successive maxima, slopes and prior results}
 \label{s:basicobjects}
 
%

 
  \subsection{Measures associated to successive maxima \`a la Chen. }
 \label{s:ChenM}
 
Let $\overline{E} = (E,(||\cdot ||_{v})_v)$ be an hermitian adelic vector bundle of rank $r = \mathrm{rank(\overline{E})} > 0$ over a number field $K$ of degree $\delta$ over $\mathbb{Q}$ (see \cite{Gaudron}, Definition $3.1$). 

\begin{definition}
\label{def:slopes}
We consider three sequences of slopes for $\overline{E}$:
\begin{itemize}
\item[i.] The (unnormalized) Harder-Narasimhan-Grayson-Stuhler slopes $(\hat{\lambda_i})_{i=1}^{r} = (\hat{\lambda_i}(\overline{E}))_{i=1}^{r}$, as defined in \cite{Gaudron}, Definition $5.10$. One has
$$
\sum_{i=1}^r \hat{\lambda_i}(\overline{E}) = \widehat{\mathrm{deg}}(\overline{E}) =r \lambda(\overline{E}),
$$
where $\widehat{\mathrm{deg}}(\overline{E})$ is the adelic degree of $\overline{E}$  (\cite{Gaudron}, Definition $4.1$), and $\lambda(\overline{E}) = \frac{1}{r} \widehat{\mathrm{deg}}(\overline{E})$ is the slope of $\overline{E}$.
\item[ii.] The naive adelic successive maxima $(\lambda_i)_{i=1}^{r} = (\lambda_i(\overline{E}))_{i=1}^{r}$ of $\overline{E}$, where   $\lambda_i(\overline{E})$ is the largest real number $\lambda$ such that the set $\overline{E}^{\geq \lambda}$ of elements of $E$ satisfying
\begin{equation}
\label{eq:heightdef}
\lambda(s) := - \sum_v \delta_v \log || s ||_{v} \geq \lambda,
\end{equation}
generates a $K$-vector space of dimension at least $i$. Here, $\delta_v$ is defined as follows, for each valuation of $v$ of $K$.
When $v$ is finite of residual characteristic $p$, if $K_v$ is the completion of $K$ at $v$, $\delta_v$ is the degree
of $K_v$ over $\mathbb{Q}_p$. When $v$ is real $\delta_v=1$, and when $v$ is complex $\delta_v=2$.
\item[iii.] The adelic successive maxima $(\lambda_i')_{i=1}^{r} = (\lambda_i'(\overline{E}))_{i=1}^{r}$ of $\overline{E}$ (see \cite{Gaudron}, Definition $5.19$) : the number $\lambda_i'(\overline{E})$ is the supremum of the quantities $- \sum_v \delta_v \log r_v $, where $(r_v)_v$ ranges over all families of positive real numbers such that the set of elements $s \in E$ satisfying
$$
\forall v, \ \ ||s||_v \leq r_v,
$$
generates a $K$-vector space of dimension at least $i$.
\end{itemize}
\end{definition}

By \cite{Gaudron}, Theorem $5.20$, one has 
$$
\sum_{i=1}^r \lambda_i'(\overline{E}) = \widehat{\mathrm{deg}}(\overline{E}) + \mathcal{O}_K(r \log(2r)).
$$
Since the same holds for the slopes $(\hat{\lambda_i}(\overline{E}))_{i=1}^{r}$, the inequalities $\hat{\lambda_i}(\overline{E}) \geq \lambda_i(\overline{E}) \geq \lambda_i'(\overline{E})$ ensure that the same estimate also holds for the slopes $(\lambda_i(\overline{E}))_{i=1}^{r}$. From this one deduce the following :

\begin{proposition}
\label{prop:hermitianadelic} Let $(\overline{E}_n)_{n \geq 1}$ be a sequence of hermitian adelic vector bundles of ranks $(r_n)_{n \geq 1}$ over $K$, such that $\log r_n = o(n)$. Assume that the sequence of probability measures
$$
\hat{\nu}_{\overline{E}_n} = \frac{1}{r_n} \sum_{i=1}^{r_n} \delta_{\frac{1}{n}\hat{\lambda_i}(\overline{E}_n)}
$$
weakly converges to some probability measure $\nu$ with compact support on $\R$. Then the sequence 
$$
\nu_{\overline{E}_n} = \frac{1}{r_n} \sum_{i=1}^{r_n} \delta_{\frac{1}{n}\lambda_i(\overline{E}_n)},
$$
weakly converges to $\nu$.
\end{proposition}

\begin{proof} Since smooth functions are dense within the space of continuous functions having compact support, it will
suffice to show that for every smooth function $h$ with compact support, 
$$
e_n = \frac{1}{r_n} \sum_{i=1}^{r_n} \left( h\left(\frac{1}{n}\hat{\lambda_i}(\overline{E}_n)\right) - h\left(\frac{1}{n}\lambda_i(\overline{E}_n)\right) \right)
$$
converges to $0$ as $n$ tends to infinity.  By the mean value theorem,
$$|e_n| \le \frac{||h'||_{\infty}}{n r_n} \sum_{i = 1}^{r_n} |\hat{\lambda_i}(\overline{E}_n) - \lambda_i(\overline{E}_n)|.$$
The discussion preceding the statement of the proposition shows
$$\sum_{i = 1}^{r_n} |\hat{\lambda_i}(\overline{E}_n) - \lambda_i(\overline{E}_n)| = \sum_{i = 1}^{r_n} \hat{\lambda_i}(\overline{E}_n) - 
\sum_{i = 1}^{r_n}  \lambda_i(\overline{E}_n) = O(r_n \log(2r_n)).$$
This gives 
$$|e_n| = O(\log(2 r_n)/n)  = o(1)$$
as claimed. 
\end{proof}

From now on, let $X$ be a projective variety of dimension $d$ over a number field $K$, and let $L$ be an ample line bundle on $X$, endowed with a continuous adelic metric $(| \cdot |_{L,v})_v$, in the sense of \cite{Zhang}. We assume that for all but a finite number of places, the metrics $(| \cdot |_{L,v})_v$ come from a single integral model of $(X,L)$ over $\mathcal{O}_K$. The $K$-vector space $H^0(X,L^{\otimes n})$ is an adelic vector bundle, in the sense of \cite{Gaudron}, if equipped with the family of norms 
$$
|| s ||_{L^{\otimes n},v} = \sup_{x \in X({\mathbb C_v})} |s(x)|^{\otimes n}_{L,v}.
$$
Even if the adelic vector bundle $H^0(X,L^{\otimes n})$ is not hermitian, one can still define its naive adelic successive maxima $(\lambda_{i,n})_{i=1}^{r_n} = (\lambda_i(H^0(X,L^{\otimes n}))_{i=1}^{r_n}$.    We will rely on the following
fundamental theorem of Chen.

\begin{theorem} {\rm (Chen)}
\label{thm:ChenTheorem} Under the above hypotheses, the sequence of probability measures
$$
\nu_{n} = \frac{1}{r_n} \sum_{i=1}^{r_n} \delta_{\frac{1}{n}\lambda_{i,n}}
$$
 converges weakly to a compactly supported probability measure $\nu$ on $\mathbb{R}$. 
 \end{theorem}
 
 \noindent Indeed, replacing the $L^{\infty}$-norms at archimedean places by $L^{2}$-norms with respect to a fixed volume form only changes the normalized successive maxima $\frac{1}{n}\lambda_{i,n}$ by the negligible amount $O \left( \frac{\log(n)}{n} \right)$, so that one is left with a sequence of hermitian adelic vector bundles over $K$ which satisfies the hypotheses of Proposition \ref{prop:hermitianadelic} by Theorem $4.1.8$ of \cite{Chen}.

\begin{remark}
\label{rem:QuentinRemark} The weak convergence of the sequence $(\nu_n)_n$ also holds when the adelic metric has mild singularities. More precisely, let us assume that $d=1$, let $(| \cdot |_{L,v})_v$ be an adelic metric coming for all but a finite number of places from a single integral model of $(X,L)$ over $\mathcal{O}_K$, and let us allow $(| \cdot |_{L,v})_v$ to have logarithmic singularities at finitely many places. Namely, if $L'$ is the line bundle $L$ endowed with a continuous adelic metric $(| \cdot |_{L',v})_v$ as above then we require $| \cdot |_{L,v} = \phi_v | \cdot |_{L',v}'$ where $\phi_v$ is a non-negative continuous function such that
$$
|1_D|_{\mathcal{O}(D)}^{M} \leq \phi_v 
$$
for some integer $M$ and some continuous metric $|\cdot|_{\mathcal{O}(D)}$ on $\mathcal{O}(D)$. The homomorphism $H^0(X,L^{\otimes n}) \rightarrow H^0(X,L'(MD)^{\otimes n})$ then has operator norm at most $1$ at each place. Thus $n^{-1} \lambda_1(H^0(X,L^{\otimes n}))$ is bounded. With the help of Corollary $4.1.4$ and Remark $4.1.5$ of \cite{Chen} one concludes that the corresponding sequence of measures $(\nu_n)_n$ is weakly convergent.
\end{remark}

\section{Modular forms and Petersson norms }
\label{s:modular}

In \S \ref{s:setupmodular} we recall some work of  Bost \cite{BostLetter} and  K\"uhn \cite{Kuhn}  concerning the interpretation of holomorphic modular
forms of weight $12k$ for $\mathrm{SL}_2(\mathbb{Z})$ as sections of the $k^{\mathrm{th}}$ power  of a particular metrized line bundle on $\mathbb{P}^1_{\mathbb{Z}}$ for $k \ge 1$.
We then  study in \S \ref{s:succminmodular} the successive maxima $\{\lambda_{i,k}\}_{i = 1}^{k}$ associated to the lattice $\mathcal{S}_{12k}(\Gamma,\mathbb{Z})$ of cusp forms
of weight $12k$ with integral Fourier coefficients with respect to the Petersson  inner product.    


\subsection{Modular forms as sections of a metrized line bundle}
\label{s:setupmodular}

  Let $\mathbb{H}$ be the upper half plane and let
$\Gamma = \mathrm{PSL}(2,\mathbb{Z})$ be  the modular group.  Then $X = \Gamma \backslash (\mathbb{H} \cup \mathbb{P}^1(\mathbb{Q}))$
has a natural structure as a Riemann surface.  The classical $j$ function of
$z \in \mathbb{H}$ has expansion
$$j(z) = \frac{1}{q} + 744 + \sum_{n=1}^\infty a_n q^n \quad \mathrm{in} \quad q = e^{2 \pi i z}.$$
The map $z \to j(z)$ defines an isomorphism $X \to \mathbb{P}^1_\mathbb{C}$.  

The volume form of the hyperbolic metric on $\mathbb{H}$ is 
\begin{equation}
\label{eq:volhyp}
\mu = \frac{dx \wedge dy}{y^2} = \frac{i}{2} \frac{dz \wedge d\overline{z}}{\mathrm{Im}(z)^2}
\end{equation}
This form has singularities at the cusp and at the elliptic fixed points of $\Gamma$, as described in \cite[\S 4.2]{Kuhn}.

Define
\begin{equation}
\label{eq:deltaexp}
\Delta(z) = q \prod_{n =1 }^\infty (1 - q^n)^{24} = q + \sum_{n > 1} b_n q^n
\end{equation}
to be  the normalized cusp form of weight $12$ for $\Gamma$. Let $S_{i\infty}$
be the unique cusp of $X$, so that $S_{i\infty}$ is associated with the orbit of
$\mathbb{P}^1(\mathbb{Q})$ under $\Gamma$.  

Suppose $k$ is a positive integer.  In \cite[Def. 4.6]{Kuhn} the
line bundle $$\mathcal{M}_{12k}(\Gamma)_\infty = \mathcal{O}_X(S_{i\infty})^{\otimes k}$$ is 
defined to be the line bundle of modular forms of weight $12k$ with respect to $\Gamma$.  This is shown to be compatible with the usual classical definition of
modular forms.  In particular,  there is an isomorphism
\begin{equation}
\label{eq:classical}
M_{12k}(\Gamma) \to H^0(X,\mathcal{O}_X(S_{i\infty})^{\otimes k})
\end{equation}
between the space $M_{12k}(\Gamma)$ of classical modular forms $f = f(z)$
of weight $12k$ and $H^0(X,\mathcal{O}_X(S_{i\infty})^{\otimes k})$ which sends $f$
to the element $f/\Delta^{k}$ of the function field $\mathbb{C}(j)$ of $X$ over $\mathbb{C}$.

The Petersson metric $| \ |_\infty$ on $\mathcal{M}_{12k}(\Gamma)_\infty$ is defined in \cite[Def. 4.8]{Kuhn} by
\begin{equation}
\label{eq:Petersson}
|f|_\infty^2 (z) = |f(z)|^2 (4 \pi \ \mathrm{Im}(z))^{12k}
\end{equation}
if $f$ is a meromorphic section of $\mathcal{M}_{12k}(\Gamma)_\infty$.  It is shown in \cite[Prop. 4.9]{Kuhn} that this metric is logarithmically singular with respect to the cusp and elliptic fixed points of $X$.  See \cite[p. 227-228]{Kuhn} for the reason that the factor $4 \pi$ is used on the right side of (\ref{eq:Petersson})

As in \cite[\S 4.11]{Kuhn}, we define an integral model of $X$ to be
$$\mathcal{X} = \mathrm{Proj}(\mathbb{Z}[Z_0,Z_1])$$ with $Z_0$ and $Z_1$ corresponding
to the global sections $j \cdot \Delta$ and $\Delta$ of the ample line bundle
$\mathcal{M}_{12}(\Gamma)_\infty$.  The point $S_{i \infty}$ 
defines a section $\overline{S}_{i \infty}$ of $\mathcal{X} = \mathbb{P}^1_{\mathbb{Z}} \to \mathrm{Spec}(\mathbb{Z})$.  We extend $\mathcal{M}_{12k}(\Gamma)_\infty$
to the line bundle 
$$\mathcal{M}_{12k}(\Gamma) = \mathcal{O}_{\mathcal{X}}(\overline{S}_{i\infty})^{\otimes k}$$
on $\mathcal{X}$.  This model then gives natural metrics $| \ |_v$ at all non-archimedean places $v$ for the induced line bundle $\mathcal{M}_{12k}(\Gamma)_{\mathbb{Q}}$ on the general fiber $X_{\mathbb{Q}} = \mathbb{Q} \otimes_{\mathbb{Z}} \mathcal{X}$. When $v$ is the infinite place of $\mathbb{Q}$, we let $| \ |_v$ be the Petersson metric $| \ |_\infty$.  

\begin{proposition}
\label{prop:globals} The global sections $H^0(\mathcal{X}, \mathcal{M}_{12k}(\Gamma))$ are identified with the $\mathbb{Z}$-lattice of all modular forms $f$
of weight $12k$ with respect to $\Gamma$ which have integral $q$-expansions
at $S_{i\infty}$.  These are the the sections $f$ of 
$H^0(\mathcal{X}_{\mathbb{Q}}, \mathcal{M}_{12k}(\Gamma)_{\mathbb{Q}})$
such that for all finite places $v$ of $\mathbb{Q}$ one has
\begin{equation}
\label{eq:finiteheight}
||f||_{\mathcal{M}_{12k}(\Gamma),v} = \mathrm{sup}_{z \in X_{\mathbb{Q}}(\mathbb{C}_v)} |f|_v(z) \le 1.
\end{equation}
If $f$ is not in $B \cdot H^0(\mathcal{X}, \mathcal{M}_{12k}(\Gamma))$ for any integer $B > 1$ then 
$$||f||_{\mathcal{M}_{12k}(\Gamma),v} = 1 \quad \mathrm{for \ all  \ finite } \quad v.$$
The sublattice $\mathcal{S}_{12k}(\Gamma,\mathbb{Z})$ of all cusp forms
in $H^0(\mathcal{X}, \mathcal{M}_{12k}(\Gamma))$ has corank $1$ and rank $k$.  If
$f \in \mathcal{S}_{12k}(\Gamma,\mathbb{Z})$, 
the $L^2$ Hermitian norm at the infinite place $v = \infty$ of  $f$ is the usual Petersson norm
\begin{equation}
\label{eq:peter}
||f||^2_{\mathcal{M}_{12k}(\Gamma),\infty,herm} = 
\int_{X(\mathbb{C})} |f|^2_\infty (z) \mu(z) = \int_{X(\mathbb{C})} |f(z)|^2 (4\pi y)^{12k} \frac{dx dy}{y^2}
\end{equation}
associated to $f$, where $\mu(z)$ is the the volume form of the hyperbolic metric given in (\ref{eq:volhyp}).  
\end{proposition}
\begin{proof} The first statement is a consequence of the fact that the $q$ expansions of $j$ and $\Delta$ have integral coefficients and begin with
$1/q$ and $q$, respectively.  The statements concerning finite places $v$ 
is just the definition of the metrics at such places which are associated to integral models of line bundles.  
The rank of $H^0(\mathcal{X}, \mathcal{M}_{12k}(\Gamma))$ over $\mathbb{Z}$ is the dimension over $\mathbb{C}$
of $H^0(X,\mathcal{M}_{12k}(\Gamma)_\infty) = H^0(X,\mathcal{O}_X(S_{i\infty})^{\otimes k})$, which equals $k+1$ by
Riemann Roch.  The last statement concerning cusp forms is the definition of the Petersson  norm
when this is normalized as in (\ref{eq:Petersson}).  
\end{proof}

\begin{remark}\label{rem:narch norm of modularforms} Since the sections $(j^{k-\ell} \Delta^k)_{\ell=0}^{k}$ form an integral basis of $H^0(\mathcal{X}, \mathcal{M}_{12k}(\Gamma))$, the norm $||\cdot||_{\mathcal{M}_{12k}(\Gamma),v}$ is given at non archimedean places by
$$
|| \sum_{\ell=0}^{k} a_\ell j^{k-\ell} \Delta^k||_{\mathcal{M}_{12k}(\Gamma),v} = \max_{0 \leq l \leq k} |a_l|_v.
$$
In particular, for any $f$ in $H^0(\mathcal{X}, \mathcal{M}_{12k}(\Gamma))_{\mathbb{Q}_v}$, the norm $||f||_{\mathcal{M}_{12k}(\Gamma),v}$ belongs to the valuation semigroup $|\mathbb{Q}_v|$.

\end{remark}

\subsection{Successive maxima and modular forms}
\label{s:succminmodular}

To state our main result we need a definition.   

\begin{definition}
\label{def:nocongruence} A non-zero form $f \in  \mathcal{S}_{12k}(\Gamma,\mathbb{Z})$  does not arise from a congruence between eigenforms  if when we write $f$ as a linear combination  $\sum_i c_i f_i$ of distinct normalized eigenforms $f_i$, the $c_i$ are algebraic integers divisible
in the ring of all algebraic integers by the g.c.d. of the Fourier coefficients of $f$. 
\end{definition}

This terminology arises from the fact that if the $c_i$ are integral but the last requirement in the definition fails, 
there is a non-trivial congruence modulo the g.c.d. of the Fourier coefficients of $f$ between the forms $f_i$.

\begin{theorem} 
\label{thm:succ} 
 Let $\{\lambda_{i,12k}\}_{i = 1}^{k}$ be the naive adelic successive maxima associated to 
$\mathcal{S}_{12k}(\Gamma,\mathbb{Z})$ in Definition \ref{def:slopes}(ii) with respect to the $L^2$ Hermitian norm defined by the Petersson norm in
(\ref{eq:peter}).  
\begin{enumerate}
\item[i.]  The sequence of probability measures
$$
\nu_{12k} = \frac{1}{k} \sum_{i=1}^{k} \delta_{\frac{1}{k}\lambda_{i,12k}}
$$
converges weakly as $k \to \infty$ to a probability measure $\nu$.  
\item[ii.]  The support of the measure $\nu$ is bounded above by $2 \pi + 6(1 - \mathrm{log}(12)) = -2.62625...$.
The support of $\nu$ is not bounded below.  
\item[iii.] As $k \to \infty$,
the proportion of successive maxima which are produced by $f \in \mathcal{S}_{12k}(\Gamma,\mathbb{Z})$
which do not arise from a congruence between eigenforms  goes
to $0$.

\end{enumerate}
\end{theorem}

This result shows that in Remark \ref{rem:QuentinRemark}, the limit measure need not have compact
support when the metrics involved are allowed to have mild singularities.
We will prove in \S \ref{s:PeterssonFourier} more quantitive results about the successive maxima $\lambda_{i,12k}$
in this Theorem.

\begin{remark}
\label{rem:HS}
Consider the divisors $\mathrm{zer}(f)$ of complex zeros of elements $f$ of 
$S = \cup_{k > 0} \ \mathcal{S}_{12k}(\Gamma,\mathbb{Z})$.  
Recall that each such $\mathrm{zer}(f) = \sum_{x \in \mathbb{P}^1(\mathbb{C})} m_x x$
defines a Dirac measure $\mu(\mathrm{zer}(f)) = \frac{1}{\mathrm{deg}(f)} \sum_{x} m_x \delta_x $.  It follows from work of Holowinsky and Soundararajan \cite[Remark 2]{HoloSound} 
and  Rudnick \cite{Rudnick} that as $f$ ranges over any sequence of non-zero Hecke eigencusp forms of weights going to infinity,
the corresponding Dirac measures  $\mu(\mathrm{zer}(f))$ converge weakly to the 
Petersson measure $\mu$ in (\ref{eq:volhyp}).  
However, due to part (iii) of Theorem \ref{thm:succ},
we cannot conclude from this much information about the measures associated to the zeros of forms with large height.  For a discussion
of the latter measures, see \S \ref{s:AnExample} and \S \ref{s:notions}.  
It would be interesting
to know whether cusp forms with integral $q$-expansions which have small Petersson norms must vanish at particular points in the upper half plane.  
\end{remark}
 
 \subsection{Petersson norms and Fourier expansions }
\label{s:PeterssonFourier}

We begin with a well known argument for bounding Petersson norms from below.

\begin{lemma}
\label{lem:estimate}
Suppose that $0 \ne f = \sum_{n = 1}^\infty a_n q^n \in \mathcal{S}_{12k}(\Gamma,\mathbb{C})$. Let $N = \mathrm{ord}_\infty(f)$.  Then 
$1 \le N \le k$, and  the $L^2$ Hermitian norm at the infinite place $v = \infty$ of  $f$ in (\ref{eq:peter})
has the property that
\begin{eqnarray}
\label{eq:lowerbound}
||f||^2_{\mathcal{M}_{12k}(\Gamma),\infty,herm} &=& \int_{X(\mathbb{C})} |f(z)|^2 (4\pi y)^{12k} \frac{dx dy}{y^2}\nonumber\\
 & \ge&  \sum_{n = 1}^\infty |a_n|^2  4 \pi e^{-4 \pi n} \frac{(12k-2)!}{n^{12k-1}}\nonumber \\
 & \ge&  |a_N|^2 \cdot 4 \pi e^{-4 \pi N} \frac{(12k-2)!}{N^{12k-1}}
\end{eqnarray}
\end{lemma}

\begin{proof}   Since the $a_n$ are in $\mathbb{C}$,
we have $\overline{f(q)} = \sum_{n = 1}^\infty \overline{a_n} \overline{q}^n.$
For a fixed $ y \ge 1$ we have (as in \cite[p. 786]{Shimura}) that
\begin{eqnarray}
\label{eq:strip}
\int_{-1/2}^{1/2} |f(x + iy)|^2 dx &=& \int_{-1/2}^{1/2} f(q) \overline{f(q)} dx\nonumber\\
&=& \int_{-1/2}^{1/2} \sum_{n, m = 1}^\infty a_n \overline{a_m} q^n \overline{q}^m \nonumber \\
&=& \sum_{n, m = 1}^\infty a_n \overline{a_m} \int_{-1/2}^{1/2} e^{2\pi i( (n - m)x + (n + m) iy)} dx\nonumber\\
&=& \sum_{n= 1}^\infty |a_n|^2 e^{-4\pi ny}
\end{eqnarray} 
The standard fundamental domain for the action of $\mathrm{SL}_2(\mathbb{Z})$ on $\mathbb{H}$
contains the set $T = \{z = x + iy: -1/2 \le x < 1/2 \quad \mathrm{and}\quad y \ge 1\}$.  Therefore
\begin{eqnarray}
\label{eq:ineq1}
\int_{X(\mathbb{C})} |f(z)|^2 (4\pi y)^{12k} \frac{dx dy}{y^2} &\ge& \int_{T} |f(z)|^2 (4\pi y)^{12k} \frac{dx dy}{y^2}\nonumber \\
&=& \int_{y = 1}^\infty \int_{x = -1/2}^{x = 1/2} |f(x + iy)|^2 dx (4\pi)^{12k} y^{12k-2} dy\nonumber \\
&=& (4 \pi)^{12k} \sum_{n = 1}^\infty |a_n|^2 \int_{y = 1}^\infty e^{-4\pi ny} y^{12k-2} dy
\end{eqnarray}

For all constants $c \ne 0$ and all integers $\ell \ge 0$, one has the indefinite integral 
\begin{equation}
\label{eq:indef}
\int e^{-cy} y^{\ell} dy = -e^{-cy} \cdot \sum_{j = 0}^{\ell} \frac{y^{\ell - j} \ \ell!}{c^{j+1} (\ell -j)!}
\end{equation}
as one sees by differentiating the right side.
Setting $c = 4\pi n$ and $\ell = 12k-2$ and then integrating the left hand side from $y =1$ to $\infty$ gives
\begin{equation}
\int_{y = 1}^\infty e^{-4\pi ny} y^{12k-2} dy = e^{-4\pi n} \sum_{j = 0}^{12k-2} \frac{(12k-2)!}{(4\pi n)^{j+1} (12k - 2 - j)!} \ge e^{-4\pi n} \frac{(12k-2)!}{(4 \pi n)^{12k-1}}
\end{equation}
Substituting this back into (\ref{eq:ineq1}) gives the claimed inequalities.

\end{proof}

\subsection{Bounds on successive maxima}

The following result will be used later to analyze the support of limit measures associated to successive maxima.  

\begin{theorem}
\label{thm:succ2} The rank of $\mathbb{S}_{12k}(\Gamma,\mathbb{Z})$ over $\mathbb{Z}$ is $k$, and $\mathbb{S}_{12k}(\Gamma,\mathbb{Z})$ has 
$\{\Delta^k j^{k - \ell}: 1 \le \ell \le k\}$ as a basis over $\mathbb{Z}$.  Suppose $0 \ne f = \sum_{n = 1}^\infty a_n q^n \in 
\mathcal{S}_{12k}(\Gamma,\mathbb{Z})$.  Let $\mathrm{ord}_\infty(f)$ be the smallest $n$ such that $a_n \ne 0$.
Then $1/k \le \mathrm{ord}_\infty(f)/k \le 1$.  
Let $\lambda(f)$ be the logarithmic height of $f$ with respect to the metrics of 
Proposition \ref{prop:globals}.  Let 
$\ell:\mathbb{R}_{>0} \to \mathbb{R}$ be the monotonically increasing function defined by 
$$\ell(c) = 2\pi c + 6(\mathrm{log}(c) + 1 - \mathrm{log}(12)).$$
\begin{enumerate}
\item[i.] For $\epsilon > 0$, there are only finitely many $k$ and $f$ for which 
$$\lambda(f)/k - \ell(\mathrm{ord}_\infty(f)/k) \ge \epsilon$$  
up to replacing $f$ by non-zero rational multiple of itself (which does not change $\lambda(f)$ or $\mathrm{ord}_\infty(f)$). 

\item[ii.]  Suppose $r_0 > \ell(1) = 2 \pi + 6(1 - \mathrm{log}(12)) = -2.62625...$.  Then for all sufficiently large $k$ and all
$f \in 
\mathcal{S}_{12k}(\Gamma,\mathbb{Z})$ one has 
$\lambda(f)/k \le r_0$.  

\item[iii.]  Suppose $1 > c > 0$ and $ \epsilon > 0$. For all sufficiently large $k$, there are at least $ck$ successive maxima $\lambda_{i,12k}$ among the
total of $k$ successive maxima associated to $\mathcal{S}_{12k}(\Gamma,\mathbb{Z})$ for which 
$$\frac{\lambda_{i,12k}}{k} \le \ell(c) + \epsilon.$$
One has $\lim_{c \to 0^+} \ell(c) = -\infty$.

\end{enumerate} 
\end{theorem}

\begin{proof} By Proposition \ref{prop:globals}, $\mathcal{S}_{12k}(\Gamma,\mathbb{Z})$ has corank $1$ in $H^0(\mathcal{X}, \mathcal{M}_{12k}(\Gamma))$.
The rank of $H^0(\mathcal{X}, \mathcal{M}_{12k}(\Gamma))$ is $k +1$, so $\mathcal{S}_{12k}(\Gamma,\mathbb{Z})$ has rank $k$. The form $\Delta^k j^{k - i} $ lies in $ \mathcal{S}_{12k}(\Gamma,\mathbb{Z})$ for $0 < i \le k$, and its first non-zero term in its Fourier expansion at $\infty$ is $q^i$.   Hence the set of these forms is a $\mathbb{Z}$-basis for $\mathcal{S}_{12k}(\Gamma,\mathbb{Z})$, and $1 \le  \mathrm{ord}_\infty(f) \le k$ for $0 \ne f \in  \mathcal{S}_{12k}(\Gamma,\mathbb{Z})$.

The logarithmic height of   $f$ with respect to the metrics $|| \ ||_{L,v}$ we have defined on $L = \mathcal{M}_{12k}(\Gamma)$ for 
each place $v$ of $\mathbb{Q}$ is
$$\lambda(f) = - \sum_{v} \mathrm{log} ||f||_{L,v}.$$
By the product formula, multiplying $f$ by a non-zero rational number does not change $\lambda(f)$.  We now replace $f$
by a rational multiple of itself without changing $\lambda(f)$ to be able to assume $f \in \mathcal{S}_{12k}(\Gamma,\mathbb{Z})$
is not in $B \cdot M_{12k}(\Gamma,\mathbb{Z})$ for any integer $B > 1$.

Proposition  \ref{prop:globals} shows 
$||f||_{L,v} = 1$ for each finite $v$, while if $v = v_\infty$ is the infinite place,
\begin{equation}
\label{eq:peter}
||f||_{L,v_\infty}^2 =  \int_{X(\mathbb{C})} |f(z)|^2 (4\pi y)^{12k} \frac{dx dy}{y^2}
\end{equation}
is the Petersson norm.  Since $f$ has integral Fourier coefficients, we find from (\ref{eq:lowerbound}) of Lemma \ref{lem:estimate} that 
 \begin{eqnarray}
 \label{eq:lower}
2 \lambda(f) &=& -  \mathrm{log}( \int_{X(\mathbb{C})} |f(z)|^2 (4\pi y)^{12k} \frac{dx dy}{y^2}) \nonumber \\
&\le& - \mathrm{log} ( 4 \pi e^{-4 \pi N} \frac{(12k-2)!}{N^{12k-1}} )\nonumber \\
&=& - \mathrm{log}(4\pi) + 4 \pi N - \mathrm{log}((12k-2)!)+ (12k-1) \mathrm{log}(N).
\end{eqnarray}

 Suppose $N = ck$ for some constants $r$ and $c$. Since $\mathrm{log}(N) \ge 0$,  (\ref{eq:lower}) gives 
\begin{eqnarray}
\label{eq:bounder}
2 \frac{\lambda(f)}{k}  &\le& - \frac{\mathrm{log}(4\pi)}{k} + 4 \pi c - \frac{\mathrm{log}((12k-2)!)}{k} + (12 - 1/k) \cdot (\mathrm{log}(c) + \mathrm{log}(k))\nonumber\\
&\le& 4 \pi c  + \frac{\mathrm{log}(12k-1) + \mathrm{log}(12k)}{k} - \frac{\mathrm{log}(12k)!}{k} + 12 \cdot (\mathrm{log}(c) + \mathrm{log}(k))\nonumber\\
&\le&  4 \pi c + 2 \mathrm{log}(12k)/k - 12 (\mathrm{log}(12) - 1) + 12 \mathrm{log}(c) \nonumber \\
\end{eqnarray}
since $\mathrm{log}((12k)!) \ge 12k \mathrm{log}(12k) - 12k$. We conclude from (\ref{eq:bounder}) that 
\begin{equation}
\label{eq:limit1} 
\frac{\lambda(f)}{k}  - \ell(c) \le \mathrm{log}(12k)/k
\end{equation}
when $\ell(c) = 2\pi c + 6(\mathrm{log}(c) + 1 - \mathrm{log}(12))$.
Thus (\ref{eq:limit1}) implies that if $\frac{\lambda(f)}{k} - \ell(c) \ge  \epsilon  > 0$ then $k$ is bounded above by a function of $\epsilon$.  For each fixed $k$,
we have $c = N/k \ge 1/k$ so $\ell(c)$ is bounded below.  Thus  $\lambda(f)/k - \ell(c)  \ge \epsilon > 0$ implies the Petersson norm
of $f$ is bounded from above. So there are only finitely many possibilities for $f$ up to multiplication by a non-zero rational number, 
as claimed in part (i) of Theorem \ref{thm:succ2}.

Part (ii) of Theorem \ref{thm:succ2} now follows from part (i). 

To prove part (iii), suppose $1 \le j \le k$.  By part (i), if $M(k,j)$ is the submodule of forms $f  \in \mathcal{S}_{12k}(\Gamma,\mathbb{Z})$ for which $\mathrm{ord}_\infty(f) > j$, the corank of $M(k,j)$ in $ \mathcal{S}_{12k}(\Gamma,\mathbb{Z})$ is
$j$.  So at least $j$ successive maxima of $ \mathcal{S}_{12k}(\Gamma,\mathbb{Z})$ do not arise from forms in $M(k,j)$.  
If $f$ is not in $M(k,j)$, then (\ref{eq:limit1}) shows
$$\frac{\lambda(f)}{k} \le \ell(\mathrm{ord}_\infty(f)/k) + \mathrm{log}(12k)/k \le \ell(j/k) + \mathrm{log}(12k)/k$$
since $\ell(c)$ is monotonically increasing with $c$.  Therefore at least $j$ of the successive maxima $\{\lambda_{i,12k}\}_{i = 1}^{k}$ associated to 
$\mathcal{S}_{12k}(\Gamma,\mathbb{Z})$ satisfy the bound
$$\frac{\lambda_{i,12k}}{k} \le \ell(j/k) + \mathrm{log}(12k)/k$$
Since  $\ell(c) \to -\infty$ as $c = j/k \to 0^+$ and $\mathrm{log}(12k)/k \to 0$ as $k \to \infty$, this proves part (iii) of Theorem \ref{thm:succ2}.
\end{proof}

\begin{lemma}\label{lem:slopebound0} There exists a constant $C > 0$ such that for any element $f$ of $\mathcal{S}_{12k}(\Gamma,\mathbb{Z})$ vanishing with order at least $N$ at infinity, we have
$$
\lambda(f) \leq  6 k \log \left( \frac{N}{k} \right) + Ck.
$$
\end{lemma}

\begin{proof} By Lemma \ref{lem:estimate} and by Stirling's formula, we have
$$
||f||^2_{\mathcal{M}_{12k}(\Gamma),\infty,herm} \geq e^{O(k)} \frac{(12k)^{12k}}{N^{12k}},
$$
hence the result by taking (opposite of) logarithms.
\end{proof}

\begin{lemma}\label{lem:slopebound1} There exists a constant $c > 0$ such that for any integers $k,\ell$ with $1 \leq \ell \leq k$, we have
$$
\lambda(\Delta^{k} j^{k-\ell}) \geq 6k \log\left( \frac{\ell}{k} \right)-c k.
$$ 
\end{lemma}

\begin{proof}
Since $\Delta^{k} j^{k-\ell}$ has integral $q$-expansion and unit leading coefficient, we have 
$$||\Delta^{k} j^{k-\ell}||_{\mathcal{M}_{12k}(\Gamma),v} = 1,$$
for any finite place $v$. In particular, we have 
$$\lambda(\Delta^{k} j^{k-\ell}) = - \log ||\Delta^{k} j^{k-\ell}||_{\mathcal{M}_{12k}(\Gamma),\infty, herm}.$$
Let 
\begin{equation}
\label{eq:fundomclosure}
F = \{z = x + iy: -1/2 \le x \le 1/2, x^2 + y^2 \ge 1\}
\end{equation} be the closure of the standard fundamental domain for the action of $\mathrm{SL}_2(\mathbb{Z})$ on $\mathbb{H}$. There is a constant $c \geq 1$ such that for any $z = x + iy$ in $F$, we have $|\Delta(z)| \leq c e^{-2 \pi y}$ and $|j(z)| \leq c e^{2 \pi y}$. We thus have
\begin{align*}
||\Delta^{k} j^{k-\ell}||^2_{\mathcal{M}_{12k}(\Gamma),\infty, herm} &= \int_{F} |\Delta(z)|^{2k}  |j(z)|^{2k - 2\ell} (4 \pi y)^{12k} \frac{dx dy}{y^2} \\
&\leq c^{4k-2 \ell} \int_{0}^{\infty} e^{- 4 \pi \ell y} (4 \pi y)^{12k} \frac{dy}{y^2} \\
&\leq 4 \pi c^{4k}  (12k-1)! \ell^{1-12k}\\
&= \left( \frac{k}{\ell} \right)^{12k} e^{O(k)},
\end{align*} 
hence the result by taking the logarithms of both sides of this inequality.
\end{proof}

\begin{lemma}\label{lem:slopebound2} Let $(\lambda_{j,k})_{j=1}^k$ be the successive maxima of $\mathcal{S}_{12k}(\Gamma, \mathbb{Q})$. We have
$$
\frac{\lambda_{j,k}}{k} = 6 \log\left( 1 - \frac{j-1}{k} \right) + O(1),
$$
where the implicit constant in $O(1)$ is absolute.
\end{lemma}

\begin{proof}
The inequality 
$$
\frac{\lambda_{j,k}}{k} \geq 6 \log\left( 1 - \frac{j-1}{k} \right) + O(1)
$$
follows from Lemma \ref{lem:slopebound1} by using the $j$ linearly independent sections $(\Delta^{k} j^{k-\ell})_{k-j+1 \leq \ell \leq k}$. We now prove the converse inequality.
Let $s_1,\dots,s_j$ be linearly independent elements of $\mathcal{S}_{12k}(\Gamma, \mathbb{Q})$ such that $\lambda(s_i) \geq \lambda_{j,k}$ for any $i$.   By Proposition \ref{prop:globals}, we can multiply the $s_i$'s by appropriate non zero rational numbers to be able to assume that
$$||s_i||_{\mathcal{M}_{12k}(\Gamma),v} = 1$$
for any finite place $v$ and for any $i$. Therefore 
$$\lambda(s_i) = - \log ||s_i||_{\mathcal{M}_{12k}(\Gamma),\infty, herm}.$$
The linear subspace of $\mathcal{S}_{12k}(\Gamma, \mathbb{Q})$ consisting of forms vanishing at $\infty$ to order at least $k - j + 2$ has dimension $j-1$, and therefore can not possibly contain all $s_i$'s. Thus there exists an index $i$ such that $s_i$ vanishes at $\infty$ to some order $N \leq k-j + 1$. By Lemma \ref{lem:estimate}, we have
$$
||s_i||_{\mathcal{M}_{12k}(\Gamma),\infty, herm} \geq  e^{-2 \pi N} \frac{(12k-2)!^{\frac{1}{2}}}{N^{6k-\frac{1}{2}}} \geq  \left( \frac{k}{N} \right)^{12k} e^{O(k)} \geq \left( \frac{k}{k-j+1} \right)^{12k} e^{O(k)}.
$$
We therefore obtain
$$
\frac{\lambda_{j,k}}{k} \leq \frac{\lambda(s_i)}{k} \leq 6 \log\left( 1 - \frac{j-1}{k} \right) + O(1).
$$

\end{proof}

\begin{lemma}
\label{lem:normcomparison}
There exists constants $c_1, c_2 >0$ such that for any element $f$ of $ \mathcal{S}_{12k}(\Gamma,\mathbb{R})$, the quantity $||f||_{\mathcal{M}_{12k}(\Gamma),\infty,sup} = \sup_{z \in X(\mathbb{C})} |f|_\infty (z) $ satisfies the inequalities
\begin{equation}
\label{eq:ineqs}
c_1 ||f||_{\mathcal{M}_{12k}(\Gamma),\infty,herm} \leq ||f||_{\mathcal{M}_{12k}(\Gamma),\infty,sup} \leq c_2 k^2 \log(3k) ||f||_{\mathcal{M}_{12k}(\Gamma),\infty,herm}
\end{equation}
\end{lemma}

\begin{proof}
One can take $c_1 = \mathrm{Vol}(X(\mathbb{C}))^{- \frac{1}{2}}$, and we therefore focus on the second inequality.  The existence of a $c_2$ for which (\ref{eq:ineqs}) holds
for a fixed $k$ follows from the fact that non-degenerate norms on a finite dimensional real vector space are comparable.  So it is enough to show that a $c_2$ exists for all sufficiently large $k$.

Let $f$ be an element of $ \mathcal{S}_{12k}(\Gamma,\mathbb{Z})$ and let $F$ be as in (\ref{eq:fundomclosure}).
Since $|f|_\infty (z) $ tends to $0$ as the imaginary part of $z \in F$ goes to infinity, there exists a point $z_0 = x_0 + i y_0$ of $F$ such that $||f||_{\mathcal{M}_{12k}(\Gamma),\infty,sup}$ is equal to $|f|_\infty (z_0)$.
Writing $f(z) = \sum_{n = 1}^\infty a_n q^n$ with $q = e^{2 i \pi n z}$, we obtain
$$
|f(z_0)| \leq \sum_{n \geq 1} |a_n| e^{-2 \pi n y_0},
$$
and then the Cauchy-Schwarz inequality yields
\begin{align}
\label{eq:cauchy}
|f(z_0)|^2 &\leq \left( \sum_{n \geq 1} a_n^2  \frac{e^{-4 \pi n}}{n^{12k-1}}  \right)  \left(  \sum_{n \geq 1}  n^{12k-1} e^{4 \pi n (1 - y_0)} \right)\nonumber \\
&\leq \frac{||f||^2_{\mathcal{M}_{12k}(\Gamma),\infty,herm}}{4 \pi (12k-2)!} \left(  \sum_{n \geq 1}  n^{12k-1} e^{4 \pi n (1 - y_0)} \right),
\end{align}
where the last inequality follows from Lemma \ref{lem:estimate}. 

Let us first assume that $y_0 \geq k \log(3k)$.  There is a positive integer $k_0$ such
that if $k \ge k_0$ and $n \ge 1$ then 
$$(12k - 1) \log(n) + 8 \pi n \le 4 \pi k \log(3k)(n-1) + 8\pi \le 4 \pi y_0 (n-1) + 8 \pi.$$
This implies 
$$
n^{12k-1} e^{4 \pi n (1 - y_0)} \leq e^{- 4 \pi y_0 + 8\pi} e^{-4 \pi n}.
$$
Therefore we can increase $k_0$, if need be, so that if $y_0 \geq k \log(3k)$ we will have for $k \ge k_0$ that  
$$
(4 \pi y_0)^{12 k}  \frac{1}{4 \pi (12k-2)!} \left(  \sum_{n \geq 1}  n^{12k-1} e^{4 \pi n (1 - y_0)} \right) \leq (4 \pi y_0)^{12 k}  e^{- 4 \pi y_0 + 8\pi}  = e^{g(y_0,k)}
$$
where $g(y_0,k) = 12k \ln(4 \pi y_0) - 4 \pi y_0 + 8 \pi$.  Using $y_0 \geq k \log(3k)$ and $k \ge k_0$ we find that $g(y_0,k) \le 0$
for $k_0$ sufficiently large.  
We thus obtain from (\ref{eq:cauchy}) that 
$$
||f||^2_{\mathcal{M}_{12k}(\Gamma),\infty,sup} = |f|_{\infty}^2(z_0)  = |f(z_0)|^2 (4 \pi y_0)^{12 k}  \leq ||f||^2_{\mathcal{M}_{12k}(\Gamma),\infty,herm}.
$$

It remains to handle the case $y_0 \leq  k \log(3k)$ and $k$ sufficiently large. We first claim that there exists a real number $R $ such that $0 < R < 1/4$ and for any $z$ in $F$, the projection from the disc $D(z,R) = \{w \in \mathbb{C}: |z - w| \le R\}$  to $X(\mathbb{C})$ is at most three to one. By a standard compactness argument, there exists a real number $R \in ]0 , \frac{1}{4}]$ such that this property holds for any $z$ in $F$ with imaginary part at most $2$ because the inertia groups in $\mathrm{PSL}_2(\mathbb{Z})$ of points of $F$ have order at most three.   It will therefore suffice to show that the projection $D(z,1/4) \to X(\mathbb{C})$ is injective if $z \in F$ has $\mathrm{Im}(z) \ge 2$. If this is not true, there is a $w \in D(z,1/4)$ such that $w \ne w' = (aw + b)/(cw + d) \in D(z,1/4)$  for some 
$\begin{pmatrix}a&b\\c&d\end{pmatrix} \in \mathrm{SL}_2(\mathbb{Z})$.  Then $\mathrm{Im}(w') = \mathrm{Im}(w)/|cw + d|^2 \ge 1$ and $\mathrm{Im}(w) \ge 1$ so we have to have $c = 0$.  But
then $w' - w$ is an integer, so $w, w' \in D(z,1/4)$ forces $w = w'$, contrary to hypothesis.

Let $R_k = k^{-1}R$.  Then $R_k < 1/4 < \sqrt{3}/2 \le y_0$  since $z_0$ is in $F$. We have
\begin{align*}
\pi R_k^2 |f(z_0)|^2 (4 \pi y_0)^{12 k} &\leq  (4 \pi y_0)^{12 k} \int_{D(z_0,R_k)} |f(z)|^2 dx dy, \\
&\leq \frac{y_0^{12k}}{(y_0 - R_k)^{12k-2}} \int_{D(z_0,R_k)} |f|_{\infty}^2(z) \frac{dx dy}{y^2} \\
&\leq 3 \frac{y_0^{12k}}{(y_0 - R_k)^{12k-2}} ||f||^2_{\mathcal{M}_{12k}(\Gamma),\infty,herm}.
\end{align*}
where the second inequality follows from $y \ge y_0 - R_k > 0$ for $y = \mathrm{Im}(z)$ and $z  \in D(z_0,R_k)$.  We therefore obtain for sufficiently large $k$ that 
$$
||f||^2_{\mathcal{M}_{12k}(\Gamma),\infty,sup} = |f|_{\infty}^2(z_0) \leq c_4 k^2 y_0^2 ||f||^2_{\mathcal{M}_{12k}(\Gamma),\infty,herm},
$$
for some absolute constant $c_4 > 0 $. Since $y_0 \leq  k \log(3k)$, this yields 
$$
||f||_{\mathcal{M}_{12k}(\Gamma),\infty,sup} \leq c_4^{\frac{1}{2}}  k^2   \log(3k)||f||_{\mathcal{M}_{12k}(\Gamma),\infty,herm}.
$$
We thus obtain the claimed inequality with $c_2 = \mathrm{max}(1, c_4^{\frac{1}{2}}  )$.
\end{proof}

\begin{lemma}
\label{lem:submultiplicativity}
There exists a real number $c$ such that for any elements $f_1, f_2$ of $ \mathcal{S}_{12k_1}(\Gamma,\mathbb{R})$ and $ \mathcal{S}_{12k_2}(\Gamma,\mathbb{R})$ respectively, we have
$$
||f_1 f_2||_{\mathcal{M}_{12(k_1 + k_2)}(\Gamma),\infty,herm} \leq e^{\psi(k_1) + \psi(k_2)} ||f_1||_{\mathcal{M}_{12k_1}(\Gamma),\infty,herm} ||f_2||_{\mathcal{M}_{12k_2}(\Gamma),\infty,herm},
$$
where $\psi(k) = 2 \log(k) + \log \log (3k)  + c$.
\end{lemma}

\begin{proof} 
Let $c_1,c_2$ be as in Lemma \ref{lem:normcomparison}. We have 
\begin{align*}
||f_1 f_2||_{\mathcal{M}_{12(k_1 + k_2)}(\Gamma),\infty,herm}  &\leq c_1^{-1} ||f_1 f_2||_{\mathcal{M}_{12(k_1 + k_2)}(\Gamma),\infty,sup}  \\
 &\leq  c_1^{-1} ||f_1||_{\mathcal{M}_{12k_1}(\Gamma),\infty,sup} ||f_2||_{\mathcal{M}_{12k_2}(\Gamma),\infty,sup} \\
 &\leq c_1^{-1} c_2^2  k_1^2 \log(3 k_1)^2  k_2^2 \log(3 k_2)^2 ||f_1||_{\mathcal{M}_{12k_1}(\Gamma),\infty,herm} ||f_2||_{\mathcal{M}_{12k_2}(\Gamma),\infty,herm},
\end{align*}
and the result follows with $c = \log(c_2) - \frac{1}{2} \log(c_1)$.

\end{proof}

\subsection{ Modified logarithmic heights}
\label{s:modloght}
To apply Chen's work in \cite{Chen} on the distribution of successive maxima, we will need some estimates for the behavior of a modification of the logarithmic height of cusp forms.

The vector space $V = \mathcal{S}_{12k}(\Gamma,\mathbb{Q})$ has a filtration defined 
by letting $V_a$ for $a \in \mathbb{R}$ be the $\mathbb{Q}$-span of all $0 \ne f \in \mathcal{S}_{12k}(\Gamma,\mathbb{Q})$
for which $\lambda(f) \ge a$.  Lemma \ref{lem:estimate} shows that $V_a = \{0\}$ if $a$ is sufficiently large.  Following Chen in \cite[p. 15, eq. (2)]{Chen}, we define a modified logarithmic height by
\begin{equation}
\label{eq:loghtdef}
\tilde{\lambda}(f) = \mathrm{sup}\{ a \in \mathbb{R}: f \in V_a\}
\end{equation}
The proof of \cite[Prop. 1.2.3]{Chen} now shows $\tilde{\lambda}(f)$ has the following properties:
\begin{lemma}
\label{lem:Chenfirst} Suppose $f $ and $g \ne -f$ are non-zero elements of $\mathcal{S}_{12k}(\Gamma,\mathbb{Q}) $.
\begin{enumerate}
\item[i.] $\tilde{\lambda}(rf) = \tilde{\lambda}(f)$ for $r \in \mathbb{Q} - \{0\}$.
\item[ii.]  $\tilde{\lambda}(f + g) \ge \mathrm{min}(\tilde{\lambda}(f),\tilde{\lambda}(g))$, with equality if
$\tilde{\lambda}(f) \ne \tilde{\lambda}(g)$
\end{enumerate}
\end{lemma}

\begin{lemma}
\label{lem:boundlemma} Let $\psi$ be as in Lemma \ref{lem:submultiplicativity}. For any elements $f_1, f_2$ of $ \mathcal{S}_{12k_1}(\Gamma,\mathbb{Q})$ and $ \mathcal{S}_{12k_2}(\Gamma,\mathbb{Q})$ respectively, we have
\begin{equation}
\label{eq:lower}
\tilde{\lambda}(f_1 f_2) \ge \tilde{\lambda}(f_1) + \tilde{\lambda}(f_2) - \psi(k_1) - \psi(k_2).
\end{equation}

\end{lemma}

\begin{proof}
Let us write $f_i = \sum_{j} g_{i,j}$, where $\lambda(g_{i,j}) \geq \tilde{\lambda}(f_i)$. For any $j_1,j_2$ and any non-archimedean place $v$, we have 
$$
||g_{1,j_1}g_{2,j_2}||_{\mathcal{M}_{12(k_1 + k_2)}(\Gamma),v} \leq ||g_{1,j_1}||_{\mathcal{M}_{12k_1}(\Gamma),v} ||g_{2,j_2}||_{\mathcal{M}_{12k_2}(\Gamma),v},
$$
and by Lemma \ref{lem:submultiplicativity}, we also have
$$
||g_{1,j_1}g_{2,j_2}||_{\mathcal{M}_{12(k_1 + k_2)}(\Gamma),\infty,herm} \leq e^{\psi(k_1) + \psi(k_2)} ||g_{1,j_1}||_{\mathcal{M}_{12k_1}(\Gamma),\infty,herm} ||g_{2,j_2}||_{\mathcal{M}_{12k_2}(\Gamma),\infty,herm}.
$$
This implies 
\begin{align*}
\lambda(g_{1,j_1}g_{2,j_2}) &\geq \lambda(g_{1,j_1}) + \lambda(g_{2,j_2}) - \psi(k_1) - \psi(k_2) \\
&\geq  \tilde{\lambda}(f_1) +  \tilde{\lambda}(f_2) - \psi(k_1) - \psi(k_2),
\end{align*}
hence the result, since $f_1 f_2 = \sum_{j_1,j_2} g_{1,j_1} g_{2,j_2}$.
\end{proof}

\subsection{Cusp forms vanishing to increasing orders at infinity}
\label{s:cuspincrease}

We study in this section the successive maxima of heights associated to cusp forms  $f \in \mathcal{S}_{12k}(\Gamma,\mathbb{Z})$
for which $\mathrm{ord}_\infty(f)$ is at least a certain positive constant times $k$.

\begin{lemma}
\label{lem:fractzero} Suppose $1 \le L, k \in \mathbb{Z}$.  The $\mathbb{Z}$-lattice $B(12k,L)$ all 
$f \in \mathcal{S}_{12k}(\Gamma,\mathbb{Z})$ for which $\mathrm{ord}_\infty(f) \ge k/L$ is the free $\mathbb{Z}$-module
with basis $\{\Delta^k j^{k - \ell}: k/L \le \ell \le k\}$.  One has
\begin{equation}
\label{eq:ranks2}
\quad k (1 - 1/L) < \mathrm{rank}_{\mathbb{Z}} (B(12k,L)) = k +1 - \lceil k/L \rceil \le k (1 - 1/L) + 1
\end{equation}
\end{lemma} 

\begin{proof} This is clear from the fact that
$\Delta^k j^{k - \ell}$ lies in $ \mathcal{S}_{12k}(\Gamma,\mathbb{Z})$ and its first non-zero term in its Fourier expansion at $\infty$ is $q^\ell$.
\end{proof}

\begin{lemma} 
\label{lem:Bmeasure} Let $\{\lambda_{i,12k,L}\}_{i = 1}^{k +1 - \lceil k/L\rceil}$ be the naive successive maxima associated in Definition \ref{def:slopes}(ii) to 
$B(12k,L)$  with respect to the $L^2$ Hermitian norm defined by the Petersson norm. 
The sequence of probability measures
\begin{equation}
\label{eq:nudef}
\nu_{12k,L} = \frac{1}{k+1 - \lceil k/L\rceil} \sum_{i=1}^{k+1 - \lceil k/L\rceil} \delta_{\frac{1}{k+1- \lceil k/L\rceil}\lambda_{i,12k,L}}
\end{equation}
converges weakly as $k \to \infty$ to a probability measure $\nu_{\infty,L}$ having compact support.   
\end{lemma}  

\begin{proof} For integers $r$ in the range $0 \le r < L$, let $B_L(r) = \oplus_{q = 0}^\infty B(12(qL+r),L)$.  If $r = 0$,
then $12(qL + r)/ L = 12q$ is an integer for all $q \ge 0$ and $B_L(0)$ is a graded algebra. It follows from Lemma \ref{lem:fractzero} that the subgroup 
$B(12qL,L) \cdot B(12q'L,L)$ of 
$B(12(q+q')L,L)$ generated by all products of elements of $B(12qL,L)$ and $B(12q'L,L)$ is equal to $B(12(q+q')L,L)$.
The work in \S \ref{s:modloght} now shows that $B_L(0)$ is integral and $\psi$-quasifiltered in the sense of
\cite[Def. 3.2.1]{Chen} with respect to the
modified logarithmic heights $\tilde{\lambda}$ on the summands of $B_L(0)$, where $\psi$ is the function from 
Lemma \ref{lem:boundlemma}.  We now observe that $\lambda_{i,12k,L}$ is the $i^{th}$ successive maxima associated to
the modified height $\tilde{\lambda}$, since $\lambda_{i,12k,L}$ is the largest real number $a$ such that the vector space 
spanned by all $f \in B(12k,L)$ with $\tilde{\lambda}(f) \ge a$ has dimension at least $i$.  
Lemma \ref{lem:slopebound0} shows that there is an upper bound independent of
$q$ on $\tilde{\lambda}_{max}(B(12qL,L))/(12qL)$ when $\tilde{\lambda}_{max}(B(12qL,L))$ is the maximal value of
$\tilde{\lambda}$ on $B(12qL,L)$.   
One can now apply \cite[Thm. 3.4.3]{Chen} to conclude that 
\begin{equation}
\label{eq:lim0}
\nu_{\infty,L} = \lim_{q \to \infty} \nu_{12qL,L}
\end{equation}
exists and has compact support when $\nu_{12k,L}$ is defined as in (\ref{eq:nudef}).  

Suppose now that $0 < r < L$.
When $k = qL + r$ and $0 \le q \in \mathbb{Z}$, $B(12k,L)$ has $\mathbb{Z}$-basis 
$b(12k,L) = \{\Delta^k j^{k - \ell}: k/L \le \ell \le k\} $.  Here $k/L = (qL + r)/L = q + r/L$ and $0< r/L < 1$, so 
$k/L \le \ell \le k$ is the same as $q+1 \le \ell \le k = qL + r$.  We have
\begin{equation}
\label{eq:firsteq}
(\Delta^{L-r} j^{L - r - 1}) \cdot (\Delta^k j^{k-\ell}) = \Delta^{L(q+1)} j^{L(q+1) - \ell - 1}
\end{equation}
since $k = qL + r$, where $ 0 \ne \Delta^{L-r} j^{L-r - 1} \in S_{12(L-r)}(\Gamma,\mathbb{Z})$.  Taking
the description of bases for $B(12k,L)$ and $B(12(q+1)L,L)$ in Lemma  \ref{lem:fractzero} into account,
we see from (\ref{eq:firsteq}) that multiplication by $(\Delta^{L-r} j^{L - r - 1})$ defines an injective homomorphism
from $B(12k,L)$ to $B(12(q+1)L,L)$.  The dimension of the cokernel of this homomorphism
is 
$$(q+1)L +1 - (q+1) - (k +1 - (q+1)) = L-r$$ 
which is bounded independently of $q$. From Lemma \ref{lem:boundlemma}, we have
\begin{eqnarray}
\tilde{\lambda}(f \cdot \Delta^{L-r} j^{L - r - 1}) &\ge& \tilde{\lambda}( f) + \tilde{\lambda}( \Delta^{L-r} j^{L - r - 1})) +c_1 \ln(k)\nonumber\\
& \ge & \tilde{\lambda}(f) + c_2 \ln(k)
\end{eqnarray}
for all $0 \ne f \in B(12k,L)$
where the constants $c_1$ and $c_2$ depend only on $L$.  It follows that for any bounded increasing continuous function $f:\mathbb{R} \to \mathbb{R}$, one has
$$\nu_{12(qL+r),L}(f) \le \nu_{12(q+1)L,L}(f) + o(1)$$
where $o(1) \to 0$ as  $q \to \infty$.  Hence
$$\limsup_{q \to\infty}  \nu_{12(qL+r),L}(f)  \le \nu_{\infty,L}(f).$$


From (\ref{eq:firsteq}) we also have
\begin{equation}
\label{eq:zeroeq}
(\Delta^k j^{k-\ell - 1}) = \Delta^{Lq} j^{Lq - \ell} \cdot (\Delta^r j^{r-1})
\end{equation}
In a similar way, this shows that multiplication by $\Delta^r j^{r-1} \in S_{12r}(\Gamma,\mathbb{Z})$ defines
an injection from $B(12Lq,L)$ to $B(12k,L)$. The dimension of the cokernel of this injection is
$r$, which is bounded independently of $q$.  By arguments similar to the one above, we obtain from (\ref{eq:zeroeq}) that
$$\liminf_{q \to\infty}  \nu_{12(qL+r),L}(f)  \ge \nu_{\infty,L}(f).$$
This completes the proof of Lemma \ref{lem:Bmeasure}. 
\end{proof}

 \subsection{Proof of parts (i) and (ii) of Theorem \ref{thm:succ}}
 \medbreak
In order to prove the weak convergence of the $\nu_{12k}$ stated in part (i) of the Proposition, we will use the limit measures $(\nu_{\infty,L})_L$ introduced in Lemma \ref{lem:Bmeasure}.  The Lipschitz norm of a bounded Lipschitz function $h:\mathbb{R} \to \mathbb{R}$ is defined to be
$$
|h|_{Lip} = \sup_{x} |h(x)| + \sup_{x \neq y} \frac{|h(x) - h(y)|}{|x-y|}.
$$

\begin{lemma}\label{lem:comparisonbound} For every pair of positive real constants $\varepsilon$ and $M$ there is a constant $L_0 = L_0(\varepsilon,M)$ for which the following
is true.  Let $h:\mathbb{R} \to \mathbb{R}$ be a bounded Lipschitz function with Lipschitz norm $|h|_{Lip} \le M$.  Suppose $L \ge L_0(\varepsilon,M)$.  Then there exists 
$k_0 = k_0(L,\varepsilon,h)$ such that for any $k \geq k_0$, we have
$$
|\nu_{12k}(h) - \nu_{12k,L}(h)| \leq \varepsilon.
$$
\end{lemma}

\begin{proof}
Let $k \geq L \geq 2$ be integers, and let $k' = k +1 - \lceil k/L \rceil $ be the rank of $B(12k,L)$. We denote by $(\lambda_{j,k})_{1 \leq j \leq k}$ and $(\lambda_{j,k,L})_{1 \leq j \leq k'}$ the successive maxima of $S_{12k}(\Gamma,\mathbb{Q})$ and $B(12k,L)$ respectively. Let us write
$$
\nu_{12k}(h) - \nu_{12k,L}(h) = S_1 + S_2 + S_3 + S_4,
$$
where we have set
\begin{align*}
S_1 &= \frac{1}{k} \sum_{k' < j \leq k} h \left( \frac{\lambda_{j,k}}{k} \right), \\
S_2 &= \frac{1}{k} \sum_{j\leq k'} \left( h \left( \frac{\lambda_{j,k}}{k} \right) - h \left( \frac{\lambda_{j,k,L}}{k} \right) \right), \\
S_3 &= \frac{1}{k} \sum_{j\leq k'} \left( h \left( \frac{\lambda_{j,k,L}}{k} \right) - h \left( \frac{\lambda_{j,k,L}}{k'} \right) \right), \\
S_4 &= \left( \frac{1}{k} - \frac{1}{k'} \right)  \sum_{j\leq k'} h \left( \frac{\lambda_{j,k,L}}{k'} \right).
\end{align*}
For $S_1$, we have the simple estimate 
$$
|S_1| \leq \frac{1}{k} \sum_{k' < j \leq k} M \leq \frac{M}{L}.
$$
A similar estimate holds for $S_4$:
$$
|S_4| \leq \frac{k - k'}{k k'} \sum_{j\leq k'} M \leq \frac{M}{L}.
$$
In order to estimate $S_3$, we first notice that an argument similar to the proof of Lemma \ref{lem:slopebound2} yields
$$
\frac{\lambda_{j,k,L}}{k} = O \left( |\log(1- \frac{k'}{k})| + 1 \right) = O( \log(3L) ).
$$
Thus there exists an absolute constant $c_1$ such that $|\lambda_{j,k,L}| \leq c_1 k \log(3L)$. This implies
\begin{align*}
|S_3| \leq \frac{M}{k} \sum_{j\leq k'} \frac{k - k'}{k k'} |\lambda_{j,k,L}| \leq \frac{c_1 M \log(3L)}{L}.
\end{align*}
It remains to estimate $S_2$. The inclusion homomorphism of $B(12k,L)$ into $S_{12k}(\Gamma,\mathbb{Q})$ preserves slopes, hence $\lambda_{j,k,L} \leq \lambda_{j,k}$ for any $j \leq k'$. We therefore have
\begin{align*}
|S_2| &\leq \frac{M}{k^2}  \sum_{j\leq k'} \left( \lambda_{j,k} - \lambda_{j,k,L}\right) \\
&= \frac{M}{k^2}\left (  \sum_{j\leq k'} \lambda_{j,k} \right ) - M \left( \frac{k'}{k} \right)^2 \nu_{12k,L}(\mathrm{id}) \\
&\leq \frac{M}{k^2}  \left ( \sum_{j\leq k'} \lambda_{j,k} \right ) - M \nu_{12k,L}(\mathrm{id}) + \frac{2c_1 M \log(3L)}{L}.
\end{align*}
Let us consider the isomorphism $S_{12k}(\Gamma,\mathbb{Q}) \rightarrow B(12(k+s),L)$ induced by multiplication by $\Delta^s$, where $s = \lfloor \frac{k-1}{L-1} \rfloor \geq 1$. 
Lemma \ref{lem:submultiplicativity} yields
$$
\lambda(\Delta^s g) \geq \lambda(g) + s \lambda(\Delta) - \psi(k) - \psi(s),
$$
for any non zero element $g$ of $S_{12k}(\Gamma,\mathbb{Q})$. Correspondingly, we have
$$
\lambda_{j,k+s,L} \geq \lambda_{j,k} + s \lambda(\Delta) - \psi(k) - \psi(s),
$$
for any $j \leq k$. We thus have
\begin{align*}
\frac{1}{k^2}  \sum_{j\leq k'} \lambda_{j,k} &\leq \frac{1}{k^2}  \sum_{j\leq k'} \lambda_{j,k+s,L} + \frac{ \psi(k)  +\psi(s)- s \lambda(\Delta)}{k} \\
&\leq \left ( \frac{1}{k^2}  \sum_{j\leq k} \lambda_{j,k+s,L} \right ) + \frac{2c_1 \log(3L)}{L} + \frac{ \psi(k)  +\psi(s)- s \lambda(\Delta)}{k} .
\end{align*}
This implies
$$
|S_2| \leq M \left(\nu_{12(k+s),L}(\mathrm{id})  -  \nu_{12k,L}(\mathrm{id})  \right) +\frac{c_2 M \log(3L)}{L},
$$
for some absolute constant $c_2$. Gathering our estimates, we obtain the existence of $L_0 = L_0(\varepsilon,M)$ such that for any $k \geq L \geq L_0$, we have
$$
|S_1| + |S_3| + |S_4| \leq \frac{\varepsilon}{2},
$$
and
$$
|S_2| \leq M \left(\nu_{12(k+s),L}(\mathrm{id})  -  \nu_{12k,L}(\mathrm{id})  \right) +\frac{\varepsilon}{4}.
$$
Since the sequence $(\nu_{12k,L}(\mathrm{id}))_k$ is convergent by Lemma \ref{lem:Bmeasure}, we further obtain the existence, for any $L \geq L_0$ of $k_0 = k_0(L,\varepsilon,M)$ such that for any $k \geq k_0$, we have $|S_2| \leq \frac{\varepsilon}{2}$, hence the result.
\end{proof}
 
\begin{cor}\label{cor:convergence}  Let $h$ be a bounded Lipschitz function from $\mathbb{R}$ to $\mathbb{R}$. Then the sequences $(\nu_{12k}(h))_k$ and $(\nu_{\infty,L}(h))_L$ are convergent and have the same limit.
\end{cor}

\begin{proof} Let $\varepsilon > 0$ be a positive real number. Let $L_0=L_0(\varepsilon,h)$ be as in Lemma \ref{lem:comparisonbound}. For any $L \geq L_0$, we have
$$
\limsup_{k \to \infty} \nu_{12k}(h) \leq \nu_{\infty,L}(h) + \varepsilon,
$$
and
$$
\liminf_{k \to \infty} \nu_{12k}(h) \geq \nu_{\infty,L}(h) - \varepsilon.
$$
In particular, we have
$$
\limsup_{k \to \infty} \nu_{12k}(h) \leq  \liminf_{k \to \infty} \nu_{12k}(h) + 2 \varepsilon.
$$
Since $\varepsilon$ is arbitrary small, this yields the convergence of the sequence $(\nu_{12k}(h))_k$. Moreover, for any $L \geq L_0$ we have
$$
|\nu_{\infty,L}(h) - \lim_{k \to \infty} \nu_{12k}(h) | \leq \varepsilon,
$$
hence the convergence of the sequence $(\nu_{\infty,L}(h))_L$ to the limit $\lim_{k \to \infty} \nu_{12k}(h)$.
\end{proof}

There exists a finite Borel measure $\nu$ on $\mathbb{R}$ such that for any continuous
function $h$ with compact support, 
\begin{equation}
\label{eq:limitf}
\nu(h) = \lim_{k \to \infty} \nu_{12k}(h) = \lim_{L\to \infty} \lim_{k \to \infty} \nu_{12k,L}(h).
\end{equation}
A limit of a weakly convergent sequence of probability measures on $\mathbb{R}$ might not be a probability measure. However, it is true in our case that the weak limit $\nu$ is a probability measure. Indeed, we have the following result, which shows that the sets of measures
$\{\nu_{12k}\}_k$ and $\{\mu_{12k,L}\}_{k \ge L}$ are uniformly tight.  

\begin{lemma}
\label{lem:further} The equalities (\ref{eq:limitf}) hold for every bounded continuous function $h$.  In particular,
$\nu$ is a probability measure, and the sequences of probability measures $\{\nu_{\infty,L}\}_L$ and $\{\nu_{12k}\}_k$
converge weakly to $\nu$.
\end{lemma}

\begin{proof} It sufficient to prove that (\ref{eq:limitf}) holds for any bounded Lipschitz function on $\mathbb{R}$. Let $h : \mathbb{R} \rightarrow \mathbb{R}$ be such a function. Let $a,b > 0$ be real numbers such that the supports of the measures $(\nu_{12k})_k$ and $(\nu_{12k,L})$ are all contained in the interval $
\mathopen]- \infty , b]$, and let $\chi : \mathbb{R} \rightarrow [0,1]$ be a continuous function with compact support, whose restriction to the interval $[-a,b]$ is equal to $1$. Lemma \ref{lem:slopebound2} implies that we have
$$
\nu_{12k,L}(\mathopen]- \infty , -a]) \leq c e^{-\frac{a}{6}},
$$
for all $k \geq L \geq 1$, where $c$ is an absolute constant. The same estimate holds as well for the measure $\nu$. In particular, we have
$$
|\nu_{12k,L}(h) - \nu_{12k,L}(\chi h)| = |\nu_{12k,L}((1-\chi)h)| \leq c ||h||_{\infty} e^{-\frac{a}{6}}.
$$
Letting $k$, and then $L$, tend to infinity, we obtain by Corollary \ref{cor:convergence} that
$$
|\lim_{k \to \infty} \nu_{12k}(h) - \nu(\chi h)| \leq c ||h||_{\infty} e^{-\frac{a}{6}}.
$$
Since we also have
$$
|\nu(h) - \nu(\chi h)|  \leq c ||h||_{\infty} e^{-\frac{a}{6}},
$$
this yields
$$
|\nu(h) - \lim_{k \to \infty} \nu_{12k}(h)|  \leq 2c ||h||_{\infty} e^{-\frac{a}{6}}.
$$
Letting $a$ tend to infinity, we obtain that the common limit of the sequences $(\nu_{12k}(h))_k$ and $(\nu_{\infty,L}(h))_L$ is $\nu(h)$, hence the result.

\end{proof}

Part (i) of Theorem \ref{thm:succ} is shown by Corollary \ref{cor:convergence}.    Part (ii) of
this Theorem concerns the support of $\nu$ now follows directly from this and  Theorem \ref{thm:succ2}.

%

\subsection{Proof of part (iii) of Theorem \ref{thm:succ}}

We suppose 
$0 \ne f \in  \mathcal{S}_{12k}(\Gamma,\mathbb{Z})$ and that $f$ does not arise from a congruence between eigenforms, in the sense of Definition \ref{def:nocongruence}.  We will develop an upper bound on $\lambda(f)$.
We have $\lambda(f) = \lambda(f/m)$ when $m$ is the g.c.d. in $\mathbb{Z}$ of the Fourier coefficients of $f$.  In view of Definition
\ref{def:nocongruence} we can replace $f$ by $f/m$ in order to be able to assume that 
\begin{equation}
\label{eq:fexpress2}
f = \sum_i c_i f_i
\end{equation}
in which the $c_i$ are non-zero algebraic algebraic integers and the $f_i$ are distinct normalized Hecke eigenforms in $\mathcal{S}_{12k}(\Gamma,\mathbb{C})$.  The Fourier coefficients of each $f_i$ are algebraic integers. Since $f$ is fixed by $\mathrm{Gal}(\overline{\mathbb{Q}}/\mathbb{Q})$, the
terms on the right side of (\ref{eq:fexpress2}) break into orbits under $\mathrm{Gal}(\overline{\mathbb{Q}}/\mathbb{Q})$ in the following sense.  If $\sigma \in \mathrm{Gal}(\overline{\mathbb{Q}}/\mathbb{Q})$
 and $f_i$ is given, then $\sigma(f_i) = f_j$ and $c_j = \sigma(c_i)$ for a unique $j$. 
 
Since the g.c.d. of the Fourier coefficients of $f$ is now $1$, we have 
\begin{equation}
\label{eq:sumform}
2 \lambda(f) = - \ln(\langle f,f\rangle )
\end{equation}
where $\langle f , f \rangle$ is the Petersson norm.  
  The Petersson inner
product $\langle f_i,f_j\rangle$ is 0 if $i$ is not $j$ since then $f_i$ and $f_j$ have distinct Hecke eigenvalues
and the Petersson inner product is Hermitian with respect to Hecke operators.  So
\begin{equation}
\label{eq:sumform2}
\langle f,f\rangle = \sum_i |c_i|^2 \langle f_i,f_i\rangle.
\end{equation}
Since each $0 \ne c_i$ is by assumption an algebraic integer, and we have shown that every Galois conjugate of $c_i$
arises as $c_j$ for some $j$, we conclude there must be an $i$ for which $|c_i| \ge 1$.   Thus
(\ref{eq:sumform2}) gives
\begin{equation}
\label{eq:answer}
\langle f,f\rangle \ge \langle f_i,f_i\rangle
\end{equation}

Recall now that since $f_i$ is a normalized eigenform, $f_i = \sum_{n = 1}^\infty a_n q^n$ has $a_1 =1$.
So $N = 1$ in Lemma \ref{lem:estimate}.  Combining  Lemma \ref{lem:estimate} with (\ref{eq:sumform}), (\ref{eq:sumform2}) and (\ref{eq:answer})  gives
\begin{equation}
\label{eq:upper}
2\lambda(f) = - \ln(\langle f,f \rangle) \le -\ln(\langle f_i,f_i \rangle ) \le -\ln( 4 \pi e^{-4 \pi} (12k-2)!)
\end{equation}
It follows that $\lambda(f)/k$ is bounded above by $-c \ln(k)$ for some constant $c$.   Since the measure $\nu$
in part (i) of  Theorem \ref{thm:succ2} is a probability measure on the real line, it follows that as $k \to \infty$
the proportion of successive maxima arising from $f$ of the above kind among all the successive maxima associated
to $\mathcal{S}_{12k}(\Gamma,\mathbb{Z})$ must go to $0$.

\section{A result from Arakelov theory} 

\label{s:ArakelovSection}

  Let $X$ be a projective   smooth curve over $K$. We assume that $X$ is geometrically irreducible, of positive genus $g$. Let $L$ be a line bundle on $X$ of degree $d \geq 2g+1$. Assume ${\mathcal X}$ is a regular model of $X$, and ${\mathcal L}$ a line bundle extending $L$ to ${\mathcal X}$. Denote by $\omega$ the relative dualizing sheaf of ${\mathcal X}$ over ${\rm Spec} ({\mathbb Z})$. Choose a positive metric $h$ on the restriction
$L_{\mathbb C}$ of $L$ to the  Riemann surface $X({\mathbb C})$ . We equip $X({\mathbb C})$ with the K\"ahler form $c_1 (L_{\mathbb C},h)$, and $\omega$ with the associated metric.

Fix a positive integer $n$. We endow
$$
{\mathcal E} = H^0 ({\mathcal X} , {\mathcal L}^{\otimes n})
$$
with the $L^2$-norm $h_{L^2}$. Let $E = H^0 (X , L^{\otimes n})$.

Given two hermitian line bundles $\bar L_1$ and $\bar L_2$ over ${\mathcal X}$, we denote by $\bar L_1 \cdot \bar L_2 \in {\mathbb R}$ the arithmetic intersection number of the first Chern classes of $\bar L_1$ and $\bar L_2$ \cite{A} \cite{D}. Let $\Delta_K$ be the absolute discriminant of $K$, and $r = [K:{\mathbb Q}]$ its absolute degree. We let $BV \in \mathbb{R} $ be the value on $X(\mathbb{C})$ of the real number defined by Bismut and Vasserot in \cite{BV}, Theorem 8 (see \cite{ARR}, p. 536).

\begin{theorem}
\label{thm:411}
Let $s \in H^0 ({\mathcal X} , {\mathcal L}^{\otimes n})$ be a nonzero global section of ${\mathcal L}^{\otimes n}$. Then
\begin{eqnarray}
\lambda (s) & \leq &n \log (n)  \frac{3rd}{4g} +      n \, \frac{\bar L^2 + d \, \bar\omega . \bar L}{2 \, gd} + nd \, \frac{\log (\vert\Delta_K\vert)}{2g}  - n \, \frac{BV}{4g} \nonumber \\
&+ &\frac{nd}g \, (r_1 + r_2) \log (2) + \frac{nrd}{2g} \, (1+\log (2\pi)) + \varepsilon,\end{eqnarray}  
where $\varepsilon $  is a function of $n$, $\bar L^2 $, $ \bar\omega . \bar L$, $ \bar\omega ^2$, $L_{\mathbb C}$, and of
 the metric $h$. 
 When  $\bar L^2 $, $ \bar\omega . \bar L$, $ \bar\omega ^2$, $L_{\mathbb C}$  and 
  $h$ are fixed,  if $n$ tends to infinity,  $\varepsilon/n$ goes to zero.
  \end{theorem}

\subsection{} To prove Theorem \ref{thm:411}, we let ${\mathbb P} (E)$ be the (Grothendieck) projective space of $E$ and $X \subset {\mathbb P} (E)$ the canonical embedding of $X$ in ${\mathbb P} (E)$. Denote by $h(X)$ the projective height of $X$. Let $N = r_n = nd + 1-g$ be the rank of $E$ and, for every $k$ between $1$ and $N$,  let $\mu_k = - \lambda_{k,n}$ be the $k$-th successive minimum of $({\mathcal E} , h_{L^2})$. Define
$$
\mu = \frac{\mu_1 + \ldots + \mu_N}N .
$$
If $C$ is the constant

\be
\label{eq1}
C = \frac{2 \, dn \, g \, (nd - 2g)}{n^2 d^2 + nd - 2g^2}
\ee
it is proved in \cite{MI}, Theorem 4, that
\be
\label{eq2}
\frac{h(X)}r + 2 \, nd \, \mu \geq C(\mu - \mu_1).
\ee

\bigskip

\subsection {} If $\overline{O(1)}$ is the restriction to $X$ of the canonical hermitian line bundle on ${\mathbb P} (E)$, the height $h(X)$ is, by definition, the number
\be
\label{eq3}
h(X) = \overline{O(1)} \cdot \overline{O(1)} .
\ee
We denote by $h_{FS}$ the metric on $L^{\otimes n}$ induced by the canonical isomorphism $L^{\otimes n} \simeq O(1)$. Let $s_1 , \ldots , s_N$ be an orthonormal basis of $(E,h_{L^2})$, and let
$$
B(x) = \sum_{j=1}^N \Vert s_j (x) \Vert^2
$$
be the Bergman kernel. For any global section $s \in H^0 (X , L^{\otimes n})$ we have
$$
\Vert s \Vert^2 = B(x) \, \Vert s \Vert_{FS}^2.
$$
Therefore, if $\varphi (x) = \log B(x)$, we get
\be
\label{eq4}
n^2 \bar L^2 = \overline{O(1)}^2 - \frac12 \int_{X({\mathbb C})} \varphi (c_1 (\overline{O(1)}) + nc_1 (\bar L))
\ee
(see, for example, \cite{BostGS} (3.2.3)). Bouche \cite{BOU} and  Tian \cite{TI} proved that, when $n$ goes to infinity,
$$
B(x) = n + \eta(x),
$$
where the function $x \to \eta(x)$ depends only on the restriction to $ {X({\mathbb C})} $ of $ \overline{O(1)}$ and 
$\bar L$.
Therefore $\varphi (x) = \log (n) + O \left( \frac1n \right)$.

\smallskip

Using (\ref{eq3}) and (\ref{eq4}) we conclude that
\be
\label{eq5}
h(X) = n^2 \, \bar L^2 + rd \, n \log (n) + O(1).
\ee

\bigskip
\subsection{}Let $\widehat\deg \, (\bar{\mathcal E})$ be the arithmetic degree of $\bar{\mathcal E} = ({\mathcal E} , h_{L^2})$, and $r_1$ (resp. $r_2$) the number of real (resp. complex) places of $K$. The second Minkowski theorem, extended to number fields by Bombieri and Vaaler, says that
\be
\label{eq6}
r \, \mu \leq - \frac{\widehat\deg \, (\bar{\mathcal E})}N + C(N,K),
\ee
where
$$
C(N,K) = \frac{\log (\vert\Delta_K\vert)}2 + (r_1 + r_2) \log (2) - \frac1N \, (r_1 \log V_N + r_2 \log (V_{2N}))
$$
$V_N$ being the volume of the standard euclidean unit ball in ${\mathbb R}^N$.

\smallskip

By the Stirling formula, when $N$ goes to infinity,
$$
\log (V_N) = - \frac N2 \log (N) + \frac N2 \, (1+\log (2\pi)) + O(\log (N)).
$$
Therefore
$$
-\frac1N \, (r_1 \log (V_N) + r_2 \log (V_{2N})) = r \left( \frac{\log (N)}2 - \frac{1+\log (2\pi)}2  + O(\log(N)/N)  \right),
$$
with an absolute  constant appearing in $O(\log(N)/N$.
Thus
\begin{eqnarray}
\label{eq7}
C(N,K) &= &\frac{\log  (\vert \Delta_K \vert)}2 + (r_1 + r_2) \log (2) \nonumber \\
&+ &r \left( \frac{\log (N)}2 - \frac{1+\log (2\pi)}2 + o(N) \right) .
\end{eqnarray}

\bigskip

\subsection{}According to \cite{ARR}, Theorem 8, as $n$ goes to infinity
$$
\widehat\deg \, (\bar{\mathcal E}) = \frac{n^2 \, \bar L^2}2 + n \left( - \frac{\bar\omega \cdot \bar L}2 + \frac{BV}4 \right) + \frac{rd}4 \, n \log (n) +  o(n),
$$
where $o(n)$  depends only on the restriction to $ {X({\mathbb C})} $ of $ \overline{O(1)}$ and 
$\bar L$.
Using (\ref{eq6}) we get
\begin{eqnarray}
\label{eq8}
n^2 \, \bar L^2 + 2 \, nd \, r \, \mu &\leq &n^2 \, \bar L^2 - \frac{2 \, nd}{nd + 1 - g} \nonumber \\
&&\left( \frac{n^2 \, \bar L^2}2 + \frac{rd}4 \, n \log (n) + n \left( - \frac{\bar\omega \cdot \bar L}2 + \frac{BV}4 \right) + o(n) \right) \nonumber \\
&+ &2 \, nd \, C(N,K) \nonumber \\
&= & - \frac{n(g-1)}d \, \bar L^2 + n \, \bar\omega \cdot \bar L - \frac{n \, BV}{2} - \frac{rd}2 \, n \log (n)  \nonumber \\
&+ &2 \, nd \, C(N,K) + o(n) \, .
\end{eqnarray}

From (\ref{eq1}) we deduce that, when $n\geq g$,
$$
C \geq  2g(1-(2g+1)/(nd)) .
$$
Therefore (\ref{eq2}) and (\ref{eq5}) imply that
$$
r \, \mu_1 \geq - \frac{n^2 \, \bar L^2}{2 \, nd} - \frac{n^2 \, \bar L^2}{2g} - \frac{rd \, n \, \mu}g + \frac{rd}{2g} \, n \log (n) + o(n),
$$
and, using (\ref{eq8}) and (\ref{eq7}), we deduce that
\begin{eqnarray}
\label{eq9}
r \, \mu_1 &\geq &n \, \frac{- \bar L^2 - d \bar\omega \cdot \bar L}{2 \, gd} -nd \, \frac{\log (|\Delta_K|)}{2g} - \frac{3rd}{4g} \, n \log (n) + \frac{n \, BV}{4g} \nonumber \\
&- &\frac{nd}g \, (r_1 + r_2) \log (2) - \frac{rd \, n}{2g} \, (1+\log (2\pi)) + o(n). 
\end{eqnarray}
Theorem \ref{thm:411} follows.

\bigskip

\subsection{}One can also get an upper bound for $r\mu_1$ when $g > 1$ as follows. Since $\mu_k \leq \mu_{k+1}$ we have
$$
r\mu_1 \leq r\mu \leq
 - \frac{\widehat\deg \, (\bar{\mathcal E})}{N} + C\, (N,  K) = - \, n \, \frac{\bar L^2}{2d} + O(\log n).
$$
The difference between this upper bound of $r\mu_1$ with its lower bound ({\ref{eq9}) is bounded from below because of the following lemma.

\begin{lemma}
\label{lem:451} Suppose $g > 1$. Then 
$$
- \, n \, \frac{\bar L^2}{2d} + n \, \frac{\bar L^2 + d \, \bar\omega \, \bar L}{2gd} \geq \frac{ nd}{8 \, g \, (g-1)} \, \bar\omega^2.
$$
\end{lemma}

\subsection{}To prove Lemma \ref{lem:451} we note that the hermitian line bundle $d \, \bar\omega - 2 \, (g-1) \, \bar L$ has degree zero on $X$. Therefore, by the Hodge index theorem of Faltings and Hriljac, its square is non positive:
$$
0 \geq (d \, \bar\omega - 2 \, (g-1) \, \bar L)^2
,
$$
i.e.
$$
- (g-1) \, \bar L^2 + d \, \bar\omega \, \bar L \geq \frac{d^2}{4 \, (g-1)} \, \bar\omega^2
$$
and Lemma \ref{lem:451} follows.



\section{Chebyshev transforms}
\label{s:Chebyshev}

\subsection{Overview}  
 Let $X$ be a projective variety  of arbitrary dimension $d$
over a number field $K$, and let $L$ be a metrized line bundle on $X$.  We will assume that $L$ is big, in the sense
that $\mathrm{dim}_K H^0(X,L^{\otimes m}) > c \ m^d$ for some $c > 0$ and all $m >> 0$.  In this section we will
develop a Chebyshev transform method for 
obtaining 
an upper bound on the height $\lambda(s)$ defined in (\ref{eq:heightdef}).  
We need lower bounds on the sup norms $||s||_v$ as $v$ varies.  We obtain such lower bounds by considering
 the behavior of $s$ near a regular point $x \in X(K)$.  Consider the first non-vanishing  coefficient $a = a(s,x)$ 
 in a suitably defined Taylor expansion of $s$ at $x$.  This $a$ lies in $K$.  The product formula
 shows there is some place $v$ where $|a|_v$ is not too close to $0$. At this $v$ we will obtain a lower bound for $||s||_v$
 which leads to a useful lower bound for $\lambda(s)$. 
 
To illustrate the details involved in this method, let us first consider the case $d = 1$, so that $X$ is a curve.  Choosing a local
parameter $t$ for the local ring $\mathcal{O}_{X,x}$ and a local trivialization $\sigma_x$ of the stalk $L_x$, we find
that $s$ has a local expansion at $x$ given by 
$$s_x  = (\sum_{n = \mathrm{ord}_x(s)}^\infty a_n t^n )\cdot \sigma_x$$
where $a_n \in K$ and $a_{\mathrm{ord}_x(s)} \ne 0$.  Here the $a_n$ depend on the choice of $\sigma_x$,
but $\mathrm{ord}_x(s)$ does not.  

The integer $\alpha = \mathrm{ord}_x(s)$ lies in the interval $[0, \mathrm{deg}(L)]$.  
To bound 
$$\lambda(s) := - \sum_v k_v \log || s ||_{v}$$
from above, we define the local Chebychev constant $c_{L,v}^{x,\sigma,t}(\alpha)$ to be the supremum over all non-zero sections $s$ of
$L$ with $\mathrm{ord}_x(s) = \alpha$ of 
\begin{equation}
\label{eq:super}
-  \log || s ||_{v} + \log |a_{\mathrm{ord}_x(s)}|_v  = \log \left | \frac{|a_{\mathrm{ord}_x(s)}|_v}{||s||_v} \right |
\end{equation}
This may be studied by $v$-adic analysis. 
We obtain an upper bound 
\begin{equation}
\label{eq:bound}
\lambda(s) \le \sum_v k_v c_{L,v}^{x,\sigma,t}(\alpha)    = c_{L}^{x,t}(\alpha)
\end{equation}
if $s$ is a section of $L$ vanishing to order $\alpha$ at $x$, since $ \sum_v k_v  \log |a_{\mathrm{ord}_x(s)}|_v = 0$ by the product formula.  

The function $c_L^{x,t}: [0,\mathrm{deg}(L)] \to \mathbf{R}$ defined by
$\alpha \to c_{L}^{x,t}(\alpha)$ is  a global Chebychev transform. 
Since we know that $\alpha$ lands in $[0,\mathrm{deg}(L)]$ for all
$s$, we obtain finally a bound of the form
$$\lambda(s) \le \sup_{0 \le \alpha \le \mathrm{deg}(L)} c_{L}^{x,t}(\alpha).$$

We now generalize the above approach to regular varieties $X$ of arbitrary dimension $d$
over $K$ using Okounkov bodies.  
Following Witt-Nystr$\hat{\mathrm{o}}$m \cite{Ny} and Yuan \cite{Yuan}, we take a regular point $x \in X(K)$, and $t_1,\dots,t_d \in \mathcal{O}_{X,x}$ a system of parameters of the regular local ring $\mathcal{O}_{X,x}$, which identifies the completion $\widehat{ \mathcal{O}_{X,x}}$ to the ring of power series $K [[ t_1, \dots,t_d ]]$ in $d$ variables over $K$. We also choose a local trivialization $\sigma \in L_{x}$ of $L$ around $x$. \\

Any section $s \in H^0(X,L^{\otimes n})$ has a germ at $x$ in $\widehat{L_x}^{\otimes n} = L_x^{\otimes n} \otimes_{ \mathcal{O}_{X,x}} \widehat{ \mathcal{O}_{X,x}} = \widehat{ \mathcal{O}_{X,x}} \sigma^{\otimes n}$, which can be uniquely written as a a power series
$$
s_x = \left( \sum_{\alpha \in \mathbb{N}^d} a_{\alpha} t^{\alpha} \right) \sigma^{\otimes n},
$$
with $a_{\alpha} \in K$. Here we have set $t^{\alpha} = t_1^{\alpha_{1}}\dots t_d^{\alpha_{d}}$. The \textbf{order of vanishing} of $s$ at $x$ is defined by the formula
$$
\mathrm{ord}_{x,t}(s) = \min \{ \alpha \in \mathbb{N}^d \ | \ a_{\alpha} \neq 0 \},
$$
where the minimum is taken with respect to the lexicographic order on $\mathbb{N}^d$ : this does not depend on $\sigma$. Likewise, we define the \textbf{leading coefficient} of $s$ at $x$ as
$$
\mathrm{lead}_{x,\sigma,t}(s) = a_{\mathrm{ord}_{x,t}(s)} \neq 0.
$$
This depends in general on the choices of $t$ and $\sigma$. \\

One strategy for upper bounding the height $\lambda(s)$ of a section $s$ is to apply the product formula
$$
1 = \prod_v | \mathrm{lead}_{x,\sigma,t}(s)|_v^{k_v},
$$
and to give an upper bound of $| \mathrm{lead}_{x,\sigma,t}(s)|_v$ in terms of $|| s ||_{L^{\otimes n},v}$, which is a problem of local nature; namely it only depends on the $v$-adic metric on $L$. This motivates the introduction of the local quantities 
$$
F_{L,v}^{x,\sigma,t}(\alpha) = \sup_{\substack{s \in H^0(X,L)_v \\ \mathrm{ord}_{x,t}(s) = \alpha }} \frac{| \mathrm{lead}_{x,\sigma,t}(s)|_v}{|| s ||_{L,v}},
$$
where $\alpha$ belongs to the finite set $\mathrm{ord}_{x,t}(H^0(X,L) \setminus \{ 0\})$. It is shown in \cite{Yuan} that the quantity
$$
c_{L,v}^{x,\sigma,t}(\alpha) = \lim_{n \rightarrow \infty} \frac{1}{n} \log F_{L^{\otimes n},v}^{x,\sigma^{\otimes n},t}(\alpha_n),
$$
where $(\alpha_n)_n$ is a sequence such that $\alpha_n \in \mathrm{ord}_{x,t}(H^0(X,L^{\otimes n}) \setminus \{ 0\})$, and such that $\frac{1}{n} \alpha_n $ converges to $\alpha$, is well-defined for any $\alpha$ in the interior of the closure $\Delta_{x,t}(L)$ of the set
$$
\bigcup_{n \geq 1} \frac{1}{n} \mathrm{ord}_{x,t}(H^0(X,L^{\otimes n}) \setminus \{ 0\}).
$$
The set $\Delta_{x,t}(L)$ is a convex body in $\R^d$ : this is the\textbf{ Okounkov body }of $L$, which depends on the choice of $t = (t_1,\dots,t_d)$. For example, if $X$ is a curve then $\Delta_{x,t}(L)$ is the interval $[0, \mathrm{deg}(L)]$. Also, if $(X,L)=(\mathbb{P}_K^d,\mathcal{O}(1))$, then $\Delta_{x,t}(L)$ is a $d$-dimensional simplex, as can be seen be reducing to the case in which $x$ is the origin of $\mathbb{A}^d_K$ and $t$ is the vector of standard coordinate functions of $\mathbb{A}^d_K$. \\

The concave function
$$
c_{L,v}^{x,\sigma,t} : \alpha \in \mathring{\Delta_{x,t}(L)} \longmapsto c_{L,v}^{x,\sigma,t}(\alpha) \in \R
$$
is called the \textbf{local Chebychev transform} of $L$ at $x$. The domain $\mathring{\Delta_{x,t}(L)}$ of $c_{L,v}$ does not depend on the metric on $L$, but $c_{L,v}^{x,\sigma,t}$ itself does. \\

\begin{example} \label{ex:centered disc} Consider the particular case $(X,L) = (\mathbb{P}_{\Q}^1,\mathcal{O}(1))$, with the line bundle metric
$$
|s([x_0:x_1])|_{L,v} = \frac{|s(x_0,x_1)|_v}{\max(|x_0|_v,r_v^{-1}|x_1|_v)},
$$
for some $r_v >0$. The maximum modulus principle if $v$ is archimedean, and a direct computation otherwise, shows that when $\infty$
is the archimedean place of $\mathbb{Q}$, we have
$$
||s||_{L^{\otimes n},\infty} = \sup_{|z|_{v} = r_v} |s(1,z)|_{\infty}.
$$
Let us consider the regular point $x = [1:0]$ with a local parameter $t = \frac{X_1}{X_0}$, and a local trivialization $\sigma = X_0$. We have
$$
F_{L^{\otimes n},v}^{x,\sigma^{\otimes n},t}(\alpha) = \sup_{s \in \C[X_0,X_1]_{n-\alpha}} \frac{| s(1,0)|_v}{|| X_1^{\alpha} s ||_{L^{\otimes n},v}},
$$
with
$$
|| X_1^{\alpha} s ||_{L^{\otimes n},v} = \sup_{|z|_{\infty} = r_v} |z^{\alpha} s(1,z)|_{\infty} = r_v^{\alpha} || s ||_{L^{\otimes n-\alpha},v},
$$
so that $F_{L^{\otimes n},v}^{x,\sigma^{\otimes n},t}(\alpha)$ equals $ r_v^{-\alpha}$. In particular, we have
$$
c_{L,v}^{x,\sigma,t}(\alpha) = - \alpha \log r_v
$$
for $\alpha \in [0,1] = \Delta_{x,t}(L)$.
\end{example}

We now define the \textbf{global Chebychev transform} as the sum
$$
c_{L}^{x,t} = \sum_{v} k_v c_{L,v}^{x,\sigma,t},
$$
which still depends on $t$, but not on the choice of the local trivialization $\sigma$ any more. While this global Chebychev transform breaks down into a sum of local components, it allows to control global invariants, such as the heights of nonzero sections :
\begin{proposition}\label{prop:boundcheby} The height of a nonzero global section $s$ of $L^{\otimes n}$ satisfies
$$
\lambda(s) \leq n \sup_{\beta \in \mathring{\Delta_x(L)}}  c_{L}^{x,t}(\beta).
$$
\end{proposition}
\begin{proof} If a section $s \in H^0(X,L^{\otimes n}) \setminus \{ 0\}$ vanishes at order $\alpha$ at $x$, then one has
$$
| \mathrm{lead}_{x,\sigma,t}(s)|_v \leq F_{L^{\otimes n},v}^{x,\sigma^{\otimes n},t}(\alpha) || s ||_{L^{\otimes n},v} \leq e^{n c_{L,v}^{x,\sigma,t}(\frac{1}{n}\alpha)}  || s ||_{L^{\otimes n},v}.
$$
Raising this inequality to the power $k_v$, and taking the product over all places $v$ yields
$$
1 = \prod_v | \mathrm{lead}_{x,\sigma,t}(s)|_v^{k_v} \leq e^{n c_{L}^{x,t}(\frac{1}{n}\alpha)} \prod_v  || s ||_{L^{\otimes n},v}^{k_v} = e^{n c_{L}^{x,t}(\frac{1}{n}\alpha) - \lambda(s)},
$$
so that $\lambda (s) \leq n c_{L}^{x,t}(\frac{1}{n}\alpha) \leq n\sup_{\beta \in \mathring{\Delta_{x,t}(L)}}  c_{L}^{x,t}(\beta)$.
\end{proof}

Likewise, a theorem of Yuan \cite{Yuan} ensures that under the hypotheses of Theorem \ref{thm:ChenTheorem}, the mean value of $c_{L}^{x,t}$ computes the expectation of the limit distribution $\nu$ appearing
in Theorem \ref{thm:ChenTheorem}:
$$
\frac{1}{\mathrm{vol}(\Delta_{x,t}(L))} \int_{\mathring{\Delta_{x,t}(L)}} c_{L}^{x,t}(\alpha) \mathrm{d}  \alpha  = \int_{\R} x \mathrm{d}   \nu.
$$
In particular, if $c_L^{x,t}$ is a constant function, then by the preceding proposition, the left hand side is an upper bound for the support of $\nu$, so that the expectation of $\nu$ is an upper bound for its support. This proves:

\begin{subcorollary} 
\label{s:diracmeasureChebychev}If the global Chebychev transform $c_L^{x,t}$ is a constant function, then the limit distribution $\nu$ is a Dirac measure supported at one point.
\end{subcorollary}

Intuitively, the limit distribution $\nu$ is expected to be completely described by $c_L^{x,t}$ when the zeroes of sections of large height concentrate at the point $x$. Since this is not the case in general (see for instance the introductory paragraph of Section \ref{s:SerreMeasures}), we should obtain better results by considering
$$
 \sup_{\substack{s \in H^0(X,L)_v \\ \mathrm{ord}_{x_1,t}(s) = \alpha_1, \dots,  \mathrm{ord}_{x_r,t}(s) = \alpha_r }} \frac{| \mathrm{lead}_{x_1,\sigma,t}(s)|_v}{|| s ||_{L,v}},
$$
where $x_1, \dots, x_r$ are distinct rational regular points (with a choice of local parameters at each of these points).


\subsection{Computation of Chebychev local transforms at archimedean places : the $L^2$ method.} Here we assume for simplicity that $X$ is a curve, i.e. $d=1$, so that $\Delta_{x,t}(L)=[0,D]$ where $D = \mathrm{deg}(L)$, and we focus on a particular archimedean place $v$. We choose a volume form $\mathrm{d} V$ on $X(\C_v)$, so that $H^0(X,L^{\otimes n})_v$ is endowed with the hermitian norm
$$
|| s ||_{L^{\otimes n},v,\mathrm{herm}}^2 = \int_{X(\C_v)} |s(x)|_{L^{\otimes n},v}^2 \mathrm{d} V(x).
$$
One can show using Gromov's lemma (see \cite[Lemma 2.7]{Yuan} and \cite[Prop. 2.13]{YuanBig}) that the Chebychev local transform $c_L^{x,t}(\alpha)$ can be computed using
$$
F_{L^{\otimes n},v, \mathrm{herm}}^{x,\sigma^{\otimes n},t}(\alpha) = \sup_{\substack{s \in H^0(X,L^{\otimes n})_v \\ \mathrm{ord}_{x,t}(s) = \alpha }} \frac{| \mathrm{lead}_{x,\sigma^{\otimes n},t}(s)|_v}{|| s ||_{L^{\otimes n},v, \mathrm{herm}}}
$$
instead of $F_{L^{\otimes n},v}^{x,\sigma^{\otimes n},t}(\alpha) $. Let us denote by $[\alpha]$ the linear form on $H^0(X,L^{\otimes n}(- \alpha x))_v$ which takes a section $s$ to the coefficient of $t^{\alpha}$ in its Taylor series expansion around $x$, so that
$$
F_{L^{\otimes n},v, \mathrm{herm}}^{x,\sigma^{\otimes n},t}(\alpha) = \sup_{s \in H^0(X,L^{\otimes n}(- \alpha x))_v } \frac{| [\alpha](s)|_v}{|| s ||_{L^{\otimes n},v, \mathrm{herm}}} 
$$
is the operator norm of $[\alpha]$ on the hermitian space $H^0(X,L^{\otimes n}(- \alpha x))_v$. In particular, if $(s_{x,\alpha,j})_j$ is an orthonormal basis of $H^0(X,L^{\otimes n}(- \alpha x))_v$, then we have
$$
F_{L^{\otimes n},v, \mathrm{herm}}^{x,\sigma^{\otimes n},t}(\alpha)^2 =  \sum_j | [\alpha](s_{x,\alpha,j})|_v^2.
$$
For $\alpha = 0$, this equals the value of the $n$-th Bergman kernel at $x$, for which precise asymptotics are known. The case $\alpha >0$ is much more elusive in general, but we will see in the remaining of this section how to handle completely the case of the Fubini-Study metric, and partially the case of the capacity metric of a disc on the projective line, by computing $F_{L^{\otimes n},v, \mathrm{herm}}^{x,\sigma^{\otimes n},t}(\alpha)$ with an explicit orthonormal basis. \\

\subsection{The $L^2$ method in use : the Chebychev local transform of the capacity metric of a disc. \label{subsubsec: L2method}} Let us consider $(X(\C_v),L) = (\mathbb{P}^1(\C_v),\mathcal{O}(1)$, with the line bundle metric
$$
|s([x_0:x_1])|_{L,v} = \frac{|s(x_0,x_1)|_v}{\max(|x_0|_v,r_v^{-1} |x_1|_v)},
$$
at an archimedean place $v$, which is the capacity metric associated to a disc of radius $r_v$ in the complex projective line, just as in example \ref{ex:centered disc}. Contrary to the situation considered in \ref{ex:centered disc}, we choose the point $x = [1:r_v]$ with a local parameter $t = \frac{X_1 - r_v X_0}{X_0}$, and a local trivialization $\sigma = X_0$. Instead of considering a volume form $\mathrm{d} V$ as above, we rather use the distribution $\mathrm{d} V$ defined by
$$
\int_{\mathbb{P}^1(\C_v)} f \mathrm{d} V = \frac{1}{4} \int_{-\pi}^{\pi} f([1:r_v e^{i \theta}]) |\sin(\theta)| \mathrm{d} \theta.
$$
By approximating this distribution by volume forms, one can check that the corresponding $F_{L^{\otimes n},v, \mathrm{herm}}$ still computes $c_{L,v}$. We now show that we have the formula
$$
F_{L^{\otimes {2n}},v, \mathrm{herm}}^{x,\sigma^{\otimes 2n},t}(2\alpha)^2 = 4^{-2\alpha}r_v^{-4\alpha}  \sum_{j=0}^{n-\alpha} (2j+2\alpha+1) \binom{j+2\alpha}{j}^2.
$$
Using Stirling's formula, this will imply the following :

\begin{proposition}\label{prop:L2} With $x,\sigma,t$ as above, the local Chebychev transform of the capacity metric associated to a disc of radius $r_v$ on the complex projective line, as defined above, with respect to a point on the boundary of the disc, is given by the formula
$$
c_{L,v}^{x,\sigma,t}(\alpha) = -\alpha \log(4r_v) + \frac{1}{2}(1+\alpha) \log(1+\alpha)-\frac{1}{2}(1-\alpha) \log(1-\alpha) - \alpha \log(\alpha).
$$
for $\alpha \in [0,1]$.
\end{proposition}

In order to compute $F_{L^{\otimes {2n}},v, \mathrm{herm}}^{x,\sigma^{\otimes 2n},t}(2\alpha)$, let us consider the orthogonal decomposition
$$
\C_v[X_0,X_1]_{2n-2k} =V^+ \oplus V^-,
$$
where $V^{\pm}$ is the space of polynomials $s \in \C_v[X_0,X_1]_{2n-2k}$ such that $s(X_0,X_1) = \pm s(r_v^{-1} X_1,r_v X_0)$. Since any $s$ in $V^-$ satisfies $s(1,r_v) = 0$, we get
$$
 F_{L^{\otimes {2n}},v, \mathrm{herm}}^{x,\sigma^{\otimes 2n},t}(2\alpha) = \sup_{s \in V^+} \frac{|s(1,r_v)|_v}{|| (X_1-r_v X_0)^{2 \alpha} s ||_{L^{\otimes n},v, \mathrm{herm}}}.
$$
However, the linear map 
$$
\Psi : T=T(Y_0,Y_1) \in \C_v[Y_0,Y_1]_{n-k} \longmapsto T(r_v X_0 X_1,r_v^2 X_0^2 + X_1^2) \in V^+
$$
is an isomorphism, with
\begin{align*}
|| (X_1-r_v X_0)^{2 \alpha} \Psi(T) ||_{L^{\otimes n},v, \mathrm{herm}}^2 &= \frac{r_v^{4n}}{4} \int_{-\pi}^{\pi} |T( 1, 2 \cos(\theta))|^2  |\sin(\theta)| |e^{i \theta} - 1|^{4 \alpha} \mathrm{d} \theta \\
&= \frac{r _v^{4n}}{4} \int_{-2}^{2} |T( 1, y)|^2 (2-y)^{2 \alpha} \mathrm{d} y,
\end{align*}
by using the substitution $y = 2 \cos(\theta)$. There is an explicit orthogonal basis of $\C_v[Y_0,Y_1]_{n-\alpha}$ for this scalar product, given by the Jacobi polynomials
$$
\mathrm{Jac}_{\alpha,j}(Y_0,Y_1) = (2Y_0 - Y_1)^{-2 \alpha} \frac{\partial^j}{j! \partial Y_1^j} Y_0^{n-\alpha-j}(2Y_0 + Y_1)^j (2Y_0 - 2Y_1)^{j + 2 \alpha}.
$$
for $0 \leq j \leq n-\alpha$. The explicit formulae
\begin{align*}
\Psi(\mathrm{Jac}_{\alpha,j})(1,r_v) &= r_v^{2n-2\alpha} (-4)^j \binom{j+2\alpha}{j},\\
|| (X_1-r_v X_0)^{2 \alpha} \Psi(\mathrm{Jac}_{\alpha,j}) ||_{L^{\otimes n},v, \mathrm{herm}} &= r_v^{2n} 4^{j+\alpha} (2j+2\alpha+1)^{-\frac{1}{2}},
\end{align*}
yield
\begin{align*}
F_{L^{\otimes {2n}},v, \mathrm{herm}}^{x,\sigma^{\otimes 2n},t}(2\alpha)^2 &= \sum_{j=0}^{n-\alpha} \frac{|\Psi(\mathrm{Jac}_{\alpha,j})(1,r_v)|_v^2}{|| (X_1-r_v X_0)^{2 \alpha} \Psi(\mathrm{Jac}_{\alpha,j}) ||_{L^{\otimes n},v, \mathrm{herm}}^2}
\\
&= 4^{-2\alpha} r_v^{-4\alpha} \sum_{j=0}^{n-\alpha} (2j+2\alpha+1) \binom{j+2\alpha}{j}^2,
\end{align*}
hence the result. 

\begin{example}\label{ex:quarterdisc}  Let us consider $(X,L) = (\mathbb{P}_{\Q}^1,\mathcal{O}(1))$, with the line bundle metric
$$
|s([x_0:x_1])|_{L,v} = \frac{|s(x_0,x_1)|_v}{\max(|x_0|_v,|x_1|_v)},
$$
for non-archimedean $v$, and 
$$
|s([x_0:x_1])|_{L,v} = \frac{|s(x_0,x_1)|_v}{\max(|x_0|_v,|4x_1 - 1|_v)},
$$
at the archimedean place. We pick the point $x = [1:0]$, with the parameter $t = \frac{x_1}{x_0}$. Then Proposition \ref{prop:L2} yields 
$$
c_{L,v}^{x,\sigma,t}(\alpha) =  \frac{1}{2}(1+\alpha) \log(1+\alpha)-\frac{1}{2}(1-\alpha) \log(1-\alpha) - \alpha \log(\alpha),
$$
which attains its maximum $\log(1+\sqrt{2}) = 0,881...$ at $\alpha = \frac{1}{\sqrt{2}}$. By Proposition \ref{prop:boundcheby} and by using Example \ref{ex:centered disc} for the archimedean places, we obtain that any nonzero global section $s$ of $\mathcal{O}(n)$ satisfies
$$
\frac{1}{n} \lambda(s) \leq \log(1+\sqrt{2}) < 0,89.
$$
On the other hand, the section
$$
s =  X_1^{34} (2X_1-X_0)^{6} (5X_1^2 - 4X_1X_0+X_0^2)^3 (29X_1^4 - 44X_1^3 X_0 +27 X_1^2 X_0^2 - 8X_1 X_0^3 + X_0^4),
$$
of $\mathcal{O}(50)$, labeled as $s_{50}$ in the introductory paragraph of Section \ref{s:SerreMeasures}, satisfies
$$
\frac{1}{n} \lambda(s) > 0,82.
$$
By taking logarithms, we obtain that for large $n$ the smallest supremum norm on the disc of radius $\frac{1}{4}$ and center $\frac{1}{4}$, of a nonzero polynomial of degree $n$ with integer coefficients is a quantity between $0,42^n$ and $0,44^n$. The change of variable $(X_0',X_1') = (X_0^2,X_1(X_0-X_1))$ yields that the corresponding quantity for a disc of radius $\frac{1}{2}$ and center $\frac{1}{2}$ is between $0,64^n$ and $0,67^n$.

\end{example}

\subsection{The $L^2$ method in use :  the Chebychev local transform of the Fubini-Study-metric.} Let us consider the complex projective space $(X(\C_v),L) = (\mathbb{P}^d(\C_v),\mathcal{O}(1))$, with the Fubini-Study metric
$$
|s([x_0:x_1])|_{L,v} = \frac{|s(x_0,x_1)|_v}{\sqrt{|x_0|_v^2 + \dots |x_d|_v^2}},
$$
at an archimedean place $v$. We pick a point $x = [x_0: \dots : x_d]$ of  $\mathbb{P}^d(\C_v)$. We have a natural identification
$$
\mathrm{T}_x \mathbb{P}^d(\C_v) = \{ T \in \C_v[X_0,\dots,X_d]_1 \ | \ T(x_0,\dots,x_d) = 0 \}.
$$ 
Let $Y_1,\dots,Y_{d+1}$ be a linear basis of $\C_v[X_0,\dots,X_d]_1$, such that $Y_1, \dots,Y_d$ span $\mathrm{T}_x \mathbb{P}^d(\C_v)$ under the identification above. The functions $t_j = \frac{Y_j}{Y_{d+1}}$, for $j=1,\dots,d$, then form a system of local parameters at $x$, while $\sigma = Y_{d+1}$ is a local trivialization of $L$ around $x$. \\
We proceed as in section \ref{subsubsec: L2method}, using the Fubini-Study volume form $dV = \frac{\omega_{FS}^d}{d!}$, where
$$
\omega_{FS} = i \partial \overline{\partial} \log(|X_0|^2+\dots+|X_d|^2).
$$
Let $U_1,\dots,U_{d+1}$ be the output of the Gram-Schmidt orthonormalization process applied to the basis $Y_1,\dots,Y_{d+1}$. In particular, $U_1,\dots,U_{d+1}$ form an orthonormal basis of $\C_v[X_0,\dots,X_d]_1$, and each coefficient
$$
\gamma_j = \langle Y_j |  U_j \rangle^{-1} 
$$
is strictly positive. Again, the functions $u_j = \frac{U_j}{U_{d+1}}$, for $j=1,\dots,d$, form a system of local parameters at $x$, and $\tau =  U_{d+1}$ is a local trivialization of $L$ around $x$. One can check the formulae
\begin{align*}
\mathrm{ord}_{x,t}(s) &= \mathrm{ord}_{x,u}(s) \\
\mathrm{lead}_{x,\sigma^{\otimes n},t}(s) &= \gamma^{\alpha} \gamma_{d+1}^{\alpha_{d+1}} \mathrm{lead}_{x,\tau^{\otimes n},u}(s) \ \text{ if } \alpha = \mathrm{ord}_{x,t}(s) \in \N^d,
\end{align*}
with $\alpha_{d+1} = 1-\sum_{j=1}^d \alpha_j$. In particular, we have,
$$
F_{L^{\otimes n},v, \mathrm{herm}}^{x,\sigma^{\otimes n},t}(\alpha) = \gamma^{\alpha} \gamma_{d+1}^{\alpha_{d+1}} F_{L^{\otimes n},v, \mathrm{herm}}^{x,\tau^{\otimes n},u}(\alpha).
$$
Since the sections $U_1^{\alpha_1} \dots U_d^{\alpha_d} U_{d+1}^{n-\alpha_1-\dots-\alpha_d} = u^{\alpha} \tau^n$, for $\alpha_1 + \dots \alpha_d \leq n$, form an orthogonal basis of the hermitian space
$$
\C_v[X_0,\dots,X_d]_n = H^0(X,L^{\otimes n}),
$$
we have by an elementary computation
$$
F_{L^{\otimes n},v, \mathrm{herm}}^{x,\tau^{\otimes n},u}(\alpha) = || U_1^{\alpha_1} \dots U_d^{\alpha_d} U_{d+1}^{n-\alpha_1-\dots-\alpha_d}||_{L^{\otimes n},v,\mathrm{herm}}^{-1} = V^{-\frac{1}{2}} \binom{n+d}{d,\alpha_1,\dots,\alpha_d}^{\frac{1}{2}},
$$
where $V = \frac{\pi^d}{d!}$ is the volume of $\mathbb{P}^d(\C_v)$ with respect to $dV$. Using Stirling's formula, we get the following result :

\begin{proposition} With $x,\sigma,t,\beta,\gamma$ as above, the Chebychev local transform of the Fubini-Study metric on the $d$-dimensional projective space is given by the formula 
$$
c_{L,v}^{x,\sigma,t}(\alpha) = \sum_{j=1}^{d+1} \alpha_i \log(\gamma_j) + \frac{1}{2} h_d(\alpha)
$$
on the Okounkov body
$$
\Delta_{x,t}(L) = \{ \alpha \in \R_+^d \ | \ \alpha_1 + \dots \alpha_d \leq 1 \},
$$
where $h_d$ is the entropy functional, defined by
$$
h_d : \alpha \in \Delta_{x,t}(L) \longmapsto \sum_{j=1}^{d+1} \alpha_j  \log \left( \frac{1}{\alpha_j}\right) \ \text{ where } \alpha_{d+1} = 1-\sum_{j=1}^d \alpha_j.
$$
\end{proposition}

 \section{Measures associated to zeros of sections }
 \label{s:SerreMeasures}
 
\subsection{An Example}
\label{s:AnExample}

Recall from \S \ref{s:ChenM} that $H^0(X,L^{\otimes n})^{\geq \lambda}$ denotes the set of of sections of $L^{\otimes n}$ of slope at least $\lambda$.
In \S \ref{s:SerreMeasures}
and \S \ref{s:Capacityone} we will study the zeros of the non-zero elements of 
 $\cup_{n = 1}^\infty H^0(X,L^{\otimes n})^{\geq \lambda}$. To motivate this we first discuss an example. 
 
Let $X = \mathbb{P}_{\mathbb{Q}}^1$ and $L = \mathcal{O}(1)$. As in Example \ref{ex:quarterdisc}, we endow $L$ with  the non archimedean metrics coming from the integral model $(\mathbb{P}_{\mathbb{Z}}^1,\mathcal{O}(1))$, and the archimedean metric given in affine coordinates by
 $$
 |s(z)|_{L,\infty} = \frac{|s(z)|}{\max ( 1 , |4z-1| )}.
 $$
 This is the capacity metric associated to the disc of center $\frac{1}{4}$ and radius $\frac{1}{4}$. For the sake of the computation, we rather use the $L^2$ metric on the boundary of this disc, rather than the supremum norm : this does not affect the asymptotic slopes. 
 
 Let $s_n$ denote a degree $n$ nonzero integer polynomial of smallest norm. A computation performed with Magma yields a small list of explicit irreducible integer polynomials $f_1,f_2,f_3,\dots$, starting with
\begin{align*}
 f_1 &= z,\\
 f_2 &= 2z-1,\\ 
 f_3 &= 5z^2 - 4z+1, \\
 f_4 &= 29z^4 - 44z^3 +27 z^2 - 8z + 1,
\end{align*}
 such that
\begin{align*}
 s_{50} &= \pm f_1^{34} f_2^{6} f_3^3 f_4,\\
 s_{100} &= \pm f_1^{63} f_2^{11} f_3^4 f_4 f_5 f_6,\\
 s_{200} &= \pm f_1^{127} f_2^{23} f_3^8 f_4^3 f_5 f_6 f_7,\\
 s_{300} &= \pm f_1^{190} f_2^{34} f_3^{12} f_4^4 f_5^2 f_6^2 f_7.
\end{align*}
 The polynomials $f_5,f_6,f_7,f_8$ have degree $6,8,8$ and $2$ respectively. Numerically, the quantity $\frac{1}{n} \mathrm{ord}_{f_1}(s_n)$ seems to converge to a limit (close to $0,63$) as $n$ grows.  Similarly,   $\lim_{n \to \infty} \frac{1}{n} \mathrm{ord}_{f_j}(s_n)$ appears to exist for higher $j$. This suggests the existence of a limit distribution of zeros associated to sections of maximal norm which is \textit{discrete}.
 
 However, replacing the disc of center $\frac{1}{4}$ and radius $\frac{1}{4}$ by a the disc of center $0$ and radius $1$, the corresponding lattices become asymptotically semistable, and one doesn't expect such a discreteness result, but rather a uniform distribution of the zeros of small sections along the boundary of the unit disk. 
 
 In \S \ref{s:notions} below we recall some work of Serre in \cite{Serre} which is useful for quantifying the intuition that the general case must interpolate between these two situations.
 \subsection{Measures}
 \label{s:notions}

  Let $Z$ be a compact
metrizable topological space.  Define $C(Z)$ to be the set of continuous real valued functions $f$ on $Z$.
A positive Radon measure on $Z$ is an $\mathbb{R}$-linear $\mathbb{R}$-valued function $\mu$ on $C(Z)$
such that $\mu(f) \ge 0$ if $f(x) \ge 0$ for all $x \in Z$.  We will sometimes write 
$\int_Z f(x) \mu(x)$ or $\int_Z f \mu$ for $\mu(f)$.  The weak topology on the space of
positive Radon measures is defined by saying $\lim_{n \to \infty} \mu_n = \mu$ if $\lim_{n \to \infty} \mu_n(f) = \mu(f)$
for all $f \in C(Z)$.  The mass of a measure $\mu$ is the value $\mu(1)$.  The space $M(Z)$ of positive
Radon measures of mass $1$ is compact for the weak topology (c.f. \cite[\S 1.1]{Serre}).  

Suppose now that $X$ is a smooth projective curve over a global field $K$.
 For each place $v$ of $K$ of $X$ we let $\mathbb{C}_v$ be the completion of an algebraic
 closure $\overline{K}_v$ of $K_v$.  
 If $v$ is archimedean, we let the topological space $Z = Z_v$ in the above discussion be $X(\mathbb{C}_v) = X(\mathbb{C})$ with the archimedean topology.
 If $v$ is non-archimedean, we let $Z = Z_v$ be the Berkovich space $X_{Berk,\mathbb{C}_v}$ described in
 \cite{Berk90}, which is compact and metrizable by \cite[\S 1]{CL1}.  There is a canonical inclusion
 of  sets $X(\mathbb{C}_v) \subset X_{Berk,\mathbb{C}_v}$.  For all $v$ we define $M_v = M(Z_v)$.

Let $v$ be an arbitrary place of $K$ and suppose $x \in X(\mathbb{C}_v)$.  If $v$ is archimedean, let $\delta_x \in M_v $
be the Dirac measure associated to $x$.  If $v$ is non-archimedean, we view $x$ as a point of $X_{Berk,\mathbb{C}_v}$
and we again let $\delta_x \in M_v $ be the associated Dirac measure.
 Suppose 
$D = \sum_{x \in X(\mathbb{C}_v)} m_x x$ is a non-zero effective divisor of $X(\mathbb{C}_v)$ that is stable under the action 
of  $\mathrm{Aut}(\mathbb{C}_v/K)$, so that $m_x = 0$ for almost all $x$.  We define the Dirac measure of $D$ to be
$$\mu(D) = \frac{1}{\mathrm{deg}(D)} \sum_{x \in X(\mathbb{C}_v)} m_x \delta_x.$$

Let $T$ be a non-empty collection of such $D$ which is closed under taking sums.  Note that $T$ is  countable.  We define $M_v(T)$ to be the closure of $\{\mu(D): D \in T\}$ in $M_v$ with respect to
the weak topology.  The argument of \cite[Prop. 1.2.2]{Serre} shows that $M_{v}(T)$ is convex and compact.  
Let $I_T$ be the set of irreducible $K$-divisors
which are components of some element of $T$.  
Suppose $S$ is a finite subset of $I_{T}$.  If $S \ne I_{T}$, define $M_v(T,S)$ be the closed convex envelope in $M_v$ of the measures $\{\mu(D): D \in I_{T} - S\}$.
If $S = I_{T}$  define $M_v(T,I_{T})$ to be the one element set consisting of the zero measure $\mu_0$ on $Z_v$.
Define
\begin{equation}
\label{eq:intersect}
M_{v}(T,\infty) = \cap_{S} \ M_{v}(T,S).
\end{equation}
where the intersection is over all finite subsets $S$ of $I_{T}$.  Then $M_{v}(T,\infty)$
is a compact, convex subset of $M_v$ if $I_T$ is infinite, and $M_v(T,\infty) = \{\mu_0\}$ if $I_T$ is finite.

The following Theorem can be proved the same way as \cite[Thm. 1.2.11]{Serre}.

\begin{theorem}
\label{thm:Serrethm}
Suppose $\mu \in M_v(T)$.  There is a unique set of non-negative real numbers $\{c_0\} \cup \{c_D: D \in I_{T}\} $ such that 
$c_0 + \sum_{D \in I_{T}} c_D = 1$ and 
\begin{equation}
\label{eq:hoopla}
\mu = \sum_{D \in I_{T}} c_D \ \mu(D) + \nu \quad \mathrm{with} \quad \nu \in c_0 \ M_{v}(T,\infty)
\end{equation}
where  $c_0 = 0$ if $I_{T}$ is finite 
\end{theorem}

In \cite{Serre}, the sum $\sum_{D \in I_{T}} c_D \ \mu(D)$ is called the atomic part $\mu_{at}$
of $\mu$, and $\nu$ is called the diffuse part of $\mu$.

We can apply these notions to the zeros of sections of a metrized line bundle $L$ on $X$ in the following way.  

\begin{definition}  
Suppose   $\lambda \in \mathbb{R} \cup \{-\infty\}$.  Let $T(L,\lambda)$ be the set of  divisors
$\mathrm{zer}(f)$ of zeros associated to non-constant elements $f$ of $ \cup_{n \ge 1} H^0(X,L^{\otimes n})^{\geq \lambda} $.  
\end{definition} 

Fix a place $v$ of $K$.  It is a natural question whether all the elements of $T(L,\lambda)$ must contain particular
irreducible divisors with at least a certain multiplicity.  We can approach this question by considering the set $M_v(T(L,\lambda))$
of measures which are limits of Dirac measures associated to $\mathrm{zer}(f)$ as above.  

Serre's Theorem \ref{thm:Serrethm}
shows that such limits will have atomic parts and diffuse parts.  The example discussed in \S \ref{s:AnExample} suggests
the following question.

\begin{question}
\label{q:askit} Fix a place $v$ of $K$.  Suppose that for each $n \ge 1$, $H^0(X,L^{\otimes n})^{\geq \lambda} $ has non-constant elements, and
$f_n \in H^0(X,L^{\otimes n})^{\geq \lambda} $
has maximal slope among all such elements.  Since $M_v(T(L,\lambda))$
is compact, there is an infinite subsequence of the measures $\{\mu(\mathrm{zer}(f_n))\}_{n \ge 1}$ which has a limit $\mu \in M_v(T(L,\lambda))$.
For all such limits $\mu$, does the atomic part $\mu_{at}$ of $\mu$ depend only on $L$?  Which measures arise
as the diffuse parts of such $\mu$?
\end{question}

In Theorems \ref{thm:succmin}  and \ref{thm:succminberk} we will show that some particular diffuse measures
arise as limit measures  $\mu$ of the sort in this question.

 \section{Adelic sets of capacity one}
\label{s:Capacityone}

\subsection{Statement of results}

In this section we assume $X$ is a smooth projective geometrically irreducible curve over a number field $K$.  Let $\overline{K}$ be an algebraic closure of $K$, and let $\mathcal{X}$
be a finite $\mathrm{Gal}(\overline{K}/K)$-stable subset of $X(\overline{K})$.  By an adelic
subset of $X$ we will mean a product $\mathcal{E} = \prod_v E_v$ over all the places $v$ of $K$
of subsets $E_v$ of $X(\overline{K}_v) - \mathcal{X}$ when
$\overline{K}_v$ is an algebraic closure of $K_v$.  As noted in at the beginning of \cite[\S 4.1]{Rumely1}, subsets of $X(\overline{K}_v)$ are better suited for global capacity theory than those of $X(\mathbb{C}_v)$.

We will assume that the $E_v$ satisfy
the standard hypotheses described in \cite[Def. 5.1.3]{Rumely1} relative to $\mathcal{X}$.  In particular, each 
$E_v$ is algebraically capacitifiable with respect to $\mathcal{X}$. We will assume each $E_v$ has positive inner capacity $\underline{\gamma}_\zeta(E_v)$ with respect to every point $\zeta \in X(\overline{K}_v) - E_v$ in the sense of \cite[p. 134-135, 196]{Rumely1}. 

In \cite[Def. 5.1.5]{Rumely1}, Rumely defined a capacity $\gamma(\mathcal{E},\mathcal{X})$ 
of such an $\mathcal{E}$ relative to $\mathcal{X}$.  For each ample effective divisor $D = \sum_{\zeta \in \mathcal{X}} a_\zeta \ \zeta$ supported on $\mathcal{X}$  one  has the sectional capacity $S_\gamma(\mathcal{E},D)$
of $\mathcal{E}$ relative to $D$ (\cite{Chinburg1}, \cite{RLV}).  We will show in 
Lemma \ref{lem:extendit} below that  Rumely's results in  \cite{RumelyDuke} imply that  $\gamma(\mathcal{E},\mathcal{X})$ is the infimum of $S_\gamma(\mathcal{E},D)^{1/\mathrm{deg}(D)^2}$ as $D$ ranges over all ample effective divisors  supported on $\mathcal{X}$ provided $\gamma(\mathcal{E},\mathcal{X}) \ge 1$.

We will recall in the next section Rumely's definition in \cite{Rumely1} of the Green's function $G(z,\zeta;E_v) \in \mathbb{R} \cup \{\infty\}$  of pairs  
$z, \zeta \in X(\overline{K}_v)$.  Define $G(z,D;E_v) = \sum_{\zeta \in \mathcal{X}} a_\zeta \ G(z,\zeta;E_v)$.   We will
regard meromorphic sections of powers of $L = \mathcal{O}_X(D)$ 
as elements of the function field $K(X)$.  Then $1$ is an element of $H^0(X,L)$
with divisor $D$.  Define a
 $v$-adic metric on $L$ via 
\begin{equation}
\label{eq:metric} 
|1|_v(z) = \mathrm{exp}(-G(z,D;E_v)) \quad \mathrm{for}\quad z \in X(\overline{K}_v)
\end{equation}
We will call these the Green's metrics on $L$ associated to $\mathcal{E}$.


We will show the following result.

\begin{theorem}
\label{thm:caponethm} Suppose that $D$ is an ample effective divisor with support $\mathcal{X}$
such that
\begin{equation}
\label{eq:capone}
\gamma(\mathcal{E},\mathcal{X}) = S_\gamma(\mathcal{E},D)^{1/\mathrm{deg}(D)^2} = 1
\end{equation}
Give $L = \mathcal{O}_X(D)$ the Green's metrics associated to $\mathcal{E}$, and suppose $E_v$ is compact if $v$ is archimedean.  
Let $\{\lambda_{n,i}\}_{i = 1}^{r_n}$ be the set of successive maxima of  $H^0(X,L^{\otimes n})$.
Let $\nu$ be the limiting distribution associated to the sets $\{\lambda_{n,i}/n\}_{i=1}^{r_n}$ as $n\to\infty$.  
Then $\nu$ is the Dirac measure supported on $0$.
\end{theorem}

Thus lattices associated
to metrized line bundles associated to adelic sets of capacity one 
are asymptotically semi-stable, in the sense that all of their successive maxima are approximately equal.


\begin{cor}
\label{ex:nicex}
Suppose there is a non-constant morphism $h:X \to \mathbb{P}^1$ over $K$ all of whose poles are at one point $\zeta \in X(K)$.  Write $\mathbb{P}^1 = \mathbb{A}^1 \cup \{\infty\}$
and $N = \mathrm{deg}(h)$, and let $D = N\zeta = h^*(\infty)$.  Suppose
$E_v = \{z \in X(\overline{K}_v): 
|h(z)|_v \le 1\}$ for all $v$.      Then the hypotheses of Theorem \ref{thm:caponethm} hold,
so that $\nu$ is the Dirac measure supported at $0$.
\end{cor}

\begin{proof}The equality (\ref{eq:capone}) in this case is a consequence of
Rumely's pullback formula \cite[Thm. 5.1.14]{Rumely1} together with the computation
of capacities of adelic disks in $\mathbb{P}^1$ given in \cite[\S 5.2]{Rumely1}. 
\end{proof}

We will discuss zeros of successive maxima in the case described in Corollary \ref{ex:nicex}.  
Identify the morphism $h:X \to \mathbb{P}^1$ of this Corollary 
with an element of the function field $K(X)$.  Let $z$ be the affine coordinate for $\mathbb{P}^1$ which has image $h$ under the induced map $K(\mathbb{P}^1) = K(z) \to K(X)$
of function fields.   

\begin{theorem}
\label{thm:succmin} Let $v$ be an archimedean place of $K$.   Let $\mu_0$ be the uniform measure on 
 the boundary of the unit disk $B_v = \{z \in \mathbb{P}^1(\overline{K}_v) = \mathbb{P}^1(\mathbb{C}): |z|_v = 1\}$.  Then  $\frac{1}{N} h^{-1}(\mu_0)$
 is the equilibrium measure $\mu(E_v,D)$ of $E_v = \overline{E}_v = h^{-1}(B_v)$ 
 in the sense of \cite[p. 214-215]{Rumely1} with respect to the polar divisor $D = N\zeta$ of $h$.  The measure $\mu(E_v,D)$  is an element of $\cap_{\lambda < 0} M_v(T(L,\lambda),\infty)$
where $M_v(T(L,\lambda),\infty)$ is the set of diffuse probability measures associated by (\ref{eq:intersect}) and Theorem \ref{thm:Serrethm} to the set $T(L,\lambda)$ of divisors of zeros of non-constant sections of 
 $\cup_{n \ge 1} H^0(X,L^{\otimes n})^{\geq \lambda}$.
\end{theorem}

We now state a version of this result for a non-archimedean place $v$ of $K$.  Define $\overline{E}_v$ to be the closure of $E_v = \{z \in X(\overline{K}_v): 
|h(z)|_v \le 1\}$ in $X_{Berk,\mathbb{C}_v}$.  Let ${\bf h}:\mathcal{X} \to \mathbb{P}^1_{\mathcal{O}_K}$ be the minimal regular model of  the morphism $h:X \to \mathbb{P}^1_K$ (see \cite{ChinburgRumely}).  Let $\overline{\infty}$ be the section of $\mathbb{P}^1_{\mathcal{O}_K} \to \mathrm{Spec}(\mathcal{O}_K)$ defined by the point at infinity. Then ${\bf h}^*(\overline{\infty}) = N\overline{\zeta} + J$
 for some vertical divisor $J$ when $\overline{\zeta}$ is the closure in $\mathcal{X}$ of the point $ \zeta \in X(K)$.   Let $\{Y_i\}_{i = 1}^\ell$ be the set of reduced irreducible components of the special fiber $\mathcal{X}_v$
 of $\mathcal{X}$, and let $m_i$ be the multiplicity of $Y_i$ in $\mathcal{X}_v$.  There is a unique point $\xi_i \in X_{Berk,\mathbb{C}_v}$ whose reduction is the generic point of $Y_i$.  Let $\delta_i$ be the delta measure supported on $\xi_i$ on $X_{Berk,\mathbb{C}_v}$, and let $({\bf h}^*(\overline{\infty}),Y_i)$ be the 
 intersection number of ${\bf h}^*(\overline{\infty})$ and $Y_i$.  Writing $D = N\zeta$, we have a measure 
\begin{equation}
 \label{eq:Berkmeasure}
 \mu(\overline{E}_v,D) = \frac{1}{N}\sum_{i = 1}^\ell  m_i ({\bf h}^*(\overline{\infty}),Y_i) \delta_i
 \end{equation}
 on $X_{Berk,\mathbb{C}_v}$.

\begin{theorem}
\label{thm:succminberk} 
Suppose $v$ is a non-archimedean place of $K$.  The measure $\mu(\overline{E}_v,D)$ in (\ref{eq:Berkmeasure}) is a probability measure 
lying in $\cap_{\lambda < 0} M_v(T(L,\lambda),\infty)$
where $M_v(T(L,\lambda),\infty)$ is the set of diffuse probability measures associated by (\ref{eq:intersect}) and Theorem \ref{thm:Serrethm} to the set $T(L,\lambda)$ of divisors of zeros of non-constant sections of 
 $\cup_{n \ge 1} H^0(X,L^{\otimes n})^{\geq \lambda}$.
\end{theorem} 

Thus under the hypotheses of Theorems \ref{thm:succmin} and \ref{thm:succminberk}, to achieve sections that demonstrate semi-stability, one can
use sections whose zeros approach the  measures
$\mu(\overline{E}_v,D)$ while avoiding any prescribed finite set of points.  
The measure in Theorem \ref{thm:succminberk}  was defined by Chambert-Loir in
\cite{CL1}, and we will use his results in the proof. 



\subsection{Green's functions in Rumely's capacity theory}
\label{s:RumelyMeasure}

 Following \cite{Rumely1}, let $q_v$ be the order of the residue field of a finite place $v$ of $K$.  If $v$ is a real place, let $q_v = e$, while if $v$ is complex let $q_v = e^2$.  Define a $v$-adic 
log by $\ln_v(r) = \ln(r)/ \ln(q_v)$ for $0 < r \in \mathbb{R}$. We let $|| \ ||_v$ be the standard absolute value $| \ |_v$ if $v$ is finite, and we let
$|| \ ||_v$ be the Euclidean absolute value if $v$ is archimedean. The product formula then becomes
$$\sum_v \ln_v ||\alpha||_v \cdot \ln(q_v) = 0$$
for $\alpha \in K - \{0\}$.  

Suppose now that $\zeta \in X(\overline{K}_v) - E_v$.  
In \cite[\S 3 - \S 4]{Rumely1} Rumely defines a real valued canonical distance function $[z,w]_\zeta$ of pairs of points $z, w \in X(\overline{K}_v) - \{\zeta\}$.  He then defines a Green's function $G(z,\zeta;E_v)$
in the following way.

Suppose first that $E_v$ is compact. Rumely shows that there is a unique positive Borel measure $\mu_v = \mu_v(E_v,\zeta)$ supported on $E_v$ that minimizes the energy integral
\begin{equation}
\label{eq:energy}
V_\zeta(E_v) = - \int_{E_v \times E_v}  \ln_v [x,w]_\zeta \ \mu_v(x) \mu_v(w)
\end{equation}
One then has a conductor potential
\begin{equation}
\label{eq:conductor}
u_E(z,\zeta) = - \int_{E_v}  \ln_v [x,w]_\zeta \ \mu_v(w).
\end{equation}
This function vanishes at almost all $z \in E_v$. One lets
\begin{equation}
\label{eq:firstGdef}
G(z,\zeta;E_v) = \left \{
  \begin{tabular}{ccc}
 $V_\zeta(E_v) - u_E(z,\zeta)  $&if & $z \not \in E_v \cup \{\zeta\}$ 
  \\
   &  &  \\
   $\infty$&if&$z = \zeta$\\
   &  &  \\
  0 & if & $z \in E_v$ 
  \end{tabular}
  \right \}
\end{equation}

Suppose now that $v$ is a finite place.    A $\mathrm{PL}_\zeta$ domain
(see \cite[Def. 4.2.6]{Rumely1}) is a subset of the form
\begin{equation}
\label{eq:PLzeta}
U_v = \{z \in X(\overline{K}_v): |f(z)|_v \le 1\}
\end{equation}
for a non-constant function $f(z) \in \overline{K}_v(X)$ having poles
only at $\zeta$.  Define
\begin{equation}
\label{eq:secondGdef}
G(z,\zeta;U_v) = \left \{
  \begin{tabular}{ccc}
 $ \frac{1}{\mathrm{deg}(f(z))} \ln_v |f(z)|_v$
   &if & $z \not \in U_v \cup \{\zeta\}$ 
  \\
   &  &  \\
   $\infty$&if&$z = \zeta$\\
  0 & if & $z \in U_v$ 
  \end{tabular}
\right \}
\end{equation}
 
Suppose now that $v$ is finite and that $E_v$ is an arbitrary
algebraically capacitifiable subset of $X(\overline{K}_v) - \{\zeta\}$ in the
sense of \cite{Rumely1}.  In \cite[\S 3 - \S 4]{Rumely1}, Rumely shows that there is
there exist an infinite increasing sequence $\{E'_{v,i}\}_{i=1}^\infty$
of compact subsets of $E_v$ and an infinite decreasing sequence
$\{U_{v,j}\}_{j = 1}^\infty$ of $PL_\zeta$ domains containing $E_v$ with the property that 
\begin{equation}
\label{eq:limit}
\lim_{i \to \infty} \gamma_\zeta(E'_{v,i}) = \gamma_\zeta(E_v) = \lim_{j \to \infty} \gamma_\zeta(U_{v,j})
\end{equation}
when $\gamma_\zeta(E_v)$ is the local capacity of $E_v$ with respect to $\zeta$.  It is shown in \cite[Thm. 4.4.4]{Rumely1} that the fact that $E_v$ is algebraically capacitable implies 
\begin{equation}
\label{eq:limgtwo}
\lim_{i \to \infty} G(z,\zeta;E'_{v,i}) = \lim_{j \to \infty} G(z,\zeta;U_{v,j})
\end{equation}
except for a set $\Delta$ of $z$ of inner capacity zero contained in $E_v$,
and the left hand limit in (\ref{eq:limgtwo}) is $0$ for all $z \in E_v$.   By \cite[Prop. 4.4.1]{Rumely1}, $G(z,\zeta;E'_{v,i})$ is non-increasing with $i$, $G(z,\zeta;U_{v,j})$ is non-decreasing with $j$, $G(z,\zeta;E'_{v,i}) \ge G(z,\zeta;U_{v,j})$ for all $i$ and $j$. The convergence in (\ref{eq:limgtwo}) is
uniform over $z$ in compact subsets of $X(\overline{K}_v) - \{\zeta\} - \Delta$.  

We now define
\begin{equation}
\label{eq:thirdGdef}
G(z,\zeta;E_v) = \left \{
  \begin{tabular}{ccc}
$ \lim_{i \to \infty} G(z,\zeta;E'_{v,i})$&if & $z \ne \zeta$ 
  \\
   &  &  \\
   $\infty$&if&$z = \zeta$\\
  \end{tabular}
  \right \}
\end{equation}

Suppose now that $D = \sum_{\zeta} n_\zeta \zeta$
an effective divisor of degree $\mathrm{deg}(D) = \sum_\zeta n_\zeta > 0$.  Let 
\begin{equation}
\label{eq:Gdef}
G(z,D;E_v) = \sum_{\zeta} n_\zeta \ G(z,\zeta;E_v) 
\end{equation}
and
\begin{equation}
\label{eq:measuregeneral}
\mu_v(E'_{v,i},D) = \frac{1}{\mathrm{deg}(D)} \sum_{\zeta} n_\zeta \ \mu_v(E'_{v,i},\zeta)
\end{equation}



\subsection{ Successive maxima for adelic sets of capacity one}
\label{s:sstable}

 The object of this section is to prove Theorem \ref{thm:caponethm}.
 We must first make a slight extension of Lemma 4.9 of \cite{CMPT}.
 
 \begin{lemma} 
 \label{lem:extendit}Suppose $\gamma(\mathcal{E},\mathcal{X}) \ge 1$. Then 
 $\gamma(\mathcal{E},\mathcal{X})$ is the infimum of $S_\gamma(\mathcal{E},D')^{1/\mathrm{deg}(D')^2}$ as $D'$ ranges over all ample effective divisors  supported on $\mathcal{X}$.
 \end{lemma}
 \begin{proof}  This result is shown in \cite[Lemma 4.9]{CMPT} if $\gamma(\mathcal{E},\mathcal{X}) > 1$.  We now suppose $\gamma(\mathcal{E},\mathcal{X}) =1$, so the
 Green's matrix $\Gamma(\mathcal{E},\mathcal{X})$ has $\mathrm{val}(\Gamma(\mathcal{E},\mathcal{X})) = 0$.  By \cite[Prop. 5.1.8, Prop. 5.1.9]{Rumely1},
 $\Gamma(\mathcal{E},\mathcal{X})$ is a symmetric real matrix with non-negative off diagonal entries that has all non-positive eigenvalues and at least one eigenvalue  equal to  $0$.  Let $I$ be the identity matrix of the same size as $\Gamma(\mathcal{E},\mathcal{X})$.  For all $\epsilon > 0$, the matrix $\Gamma_\epsilon = \Gamma(\mathcal{E},\mathcal{X}) - \epsilon I$ is negative definite, symmetric and has non-negative off diagonal entries.  We now apply the arguments of \cite[Lemma 4.9]{CMPT}
 to $\Gamma_\epsilon$ and let $\epsilon \to 0$. Since the space of probability vectors
 of a prescribed size is compact, this implies Lemma \ref{lem:extendit}.
 \end{proof}

 \begin{lemma}
 \label{lem:greensup}  Let $|\ |_v$ be the Green's metric
 (\ref{eq:metric}) on $L_v$, and let $| \ |_v^{\otimes n}$
 be the resulting metric on $L_v^{\otimes n}$.  For
  $f_v \in H^0(X,L^{\otimes n})$ and $z \in X(\overline{K}_v)$
  let $|f_v|^{\otimes n}_v(z)$ be the norm with respect to $| \ |_v^{\otimes n}$ of the image of $f_v$ in the fiber
  of $L_v^{\otimes n}$ at $z$.   Regarding $f_v$ as an element of the function field
  $K_v(X)$, let $f_v(z) \in \overline{K}_v \cup \{\infty\}$ be the value of $f_v$ at $z$.
  Then
  \begin{equation}
  \label{eq:estimate}
  ||f_v||_{L^{\otimes n},v} = \mathrm{sup}_ {z \in X(\overline{K}_v)} |f_v|^{\otimes n}(z) \quad \mathrm{equals}\quad \mathrm{sup}(f_v,E_v) = \mathrm{sup}_{z \in E_v}    |f_v(z)|_v .
  \end{equation}
  \end{lemma}
  
  \begin{proof} In view of (\ref{eq:metric} ), the Green's metric on $f_v \in H^0(X,L^{\otimes n})$ is
  specified by
  \begin{equation}
  \label{eq:metricvalue}
 \ln |f_v|_v^{\otimes n}(z) = \ln (|f_v(z)|_v) + \ln  (| 1|_v^{\otimes n}(z)) = \ln |f_v(z)|_v -  n G(z,D;E_v).
 \end{equation}
 Suppose first that $v$ is archimedean.  
 We have supposed in this case that $E_v$ is compact. Then $\ln |f_v(z)|_v - n G(z,D;E_v)$ is a well defined
 harmonic function on the open set $X(\overline{K}_v) - E_v = X(\mathbb{C}) - E_v$, so it achieves
 its maximum on the boundary of $X(\overline{K}_v) - E_v$.  This boundary lies in $E_v$
and $G(z,D;E_v) = 0$ for $z \in E_v$ by \cite[Def. 3.2.1]{Rumely1}, so (\ref{eq:estimate}) holds.  Suppose
now that $v$ is non-archimedean.   By \cite[p. 282, Def. 4.4.12]{Rumely1}, $G( z,D;E_v)$ is the supremum of $G(z,D;U_v)$ over
RL domains $U_v \supset E_v$ defined by functions having poles in $\mathcal{X}$. 
The fact that (\ref{eq:estimate}) holds is now a consequence of the formula for 
$G(z,D;U_v)$ when $U_v$ is RL-domain in (\cite[p. 277, eq. (2)]{Rumely1}) together with the maximum
modulus principle of \cite[Thm. 1.4.2]{Rumely1}.  
  \end{proof}
  
   \begin{lemma}
  \label{lem:no small}  There is no section $f \in H^0(X,L^{\otimes n})$
  that has height $\lambda(f) > 0$.
  \end{lemma}
  
  \begin{proof}  Suppose $f \in H^0(X,L^{\otimes n})$
 is a section with $\lambda(f) = - \sum_v k_v \log || f ||_{L^{\otimes n},v} > 0$.
 Then $f$ defines a morphism $X \to \mathbb{P}^1$ such that
 $f^{-1}(\infty) = D'$ is supported on $\mathcal{X}$, since $f$
 is a section of $L^{\otimes n} = \mathcal{O}_X(n D)$ and $D$ is supported on
 $\mathcal{X}$.   Write $r_v  = \mathrm{sup}(f,E_v)$.
 We let $\mathcal{E}' = \prod_v E'_v$ be the adelic polydisc of the projective line $\mathbb{P}^1$
 such that each $E'_v \subset A^1(\overline{K}_v) = P^1(\overline{K}) - \{\infty\}$
 is the disc around the origin with radius $r_v > 0$ with respect to $| \ |_v$.  By the definition of
 the $r_v$, we have $\mathcal{E} \subset f^{-1}( \mathcal{E}')$.
This and Rumely's pullback formula \cite[Prop. 4.1]{RumelyDuke} give 
 \begin{equation}
 \label{eq:slist}
 S_\gamma(\mathcal{E},D') \le S_\gamma(f^{-1}(\mathcal{E}'),D') = S_\gamma(\mathcal{E}',\infty)^{\mathrm{deg}(D')}.
 \end{equation}
 By Lemma \ref{lem:greensup}, 
 $\lambda(f) = - \sum_v k_v \log || f ||_{L^{\otimes n},v}  = - \sum_v [K_v:\mathbb{Q}_p] \ln(r_v)$
 and this is $- \ln(S_\gamma(\mathcal{E}',\infty))$ 
 by \cite[Prop. 3.1]{RumelyDuke} and \cite[p. 339]{Rumely1}.  Because $\lambda(f) > 0$ we conclude from (\ref{eq:slist}) that $S_\gamma(\mathcal{E},D') < 1$.  
Hence $S_\gamma(\mathcal{E},D')^{1/\mathrm{deg}(D')^2} < 1$.
This contradicts the hypothesis in (\ref{eq:capone}) because of Lemma \ref{lem:extendit}.
 \end{proof}

\begin{lemma}
\label{lem:shrink}  Suppose $\epsilon > 0$.  There is a finite place $v_0$ of $K$ and a subset $E'_{v_0}$ of $E_{v_0}$ with the following properties:
\begin{enumerate}
\item[1.] $E'_{v_0}$ is capacitifiable with respect to $D$, and  
$|G(z,D;E'_{v_0}) - G(z,D;E_{v_0})| < \epsilon$ for all $z \in X(\mathbb{C}_v)$.
\item[2.] The set $\mathcal{E}' = E'_{v_0} \times (\prod_{v \ne v_0} E_v)$ has capacity
$S_\gamma(\mathcal{E}',D) < S_\gamma(\mathcal{E},D) = 1$.
\item[3.] Let $\lambda(s) = - \sum_v k_v \log || s ||_{L^{\otimes n},v}$ be the height of a section
$s \in H^0(X,L^{\otimes n})$ associated to the Green's metrics for $\mathcal{E}$,
and let $\lambda'(s)$ is the corresponding height for $\mathcal{E}'$. Then
$|\lambda(s) - \lambda'(s)|/n < k_{v_0} \epsilon$.
\item[4.] There is a rational function $f \in K(X)$ whose divisor of poles is a positive integral
multiple of $D$ with the following properties: we have $\mathrm{sup}(f,E_v) \le 1$ for all finite $v \ne v_0$,
$\mathrm{sup}(f,E'_{v_0}) \le 1$ and $\mathrm{sup}(f,E_v) < 1$ for all archimedean $v$.
\end{enumerate}
\end{lemma}

\begin{proof}   Choose a place $v_0$ where $E_{v_0}$ is $\mathcal{X}$-trivial in the sense of 
\cite[Def. 5.1.1]{Rumely1}
 By \cite[Prop. 4.4.13]{Rumely1}, the Green's function $G(z,\zeta;E_{v_0})$ for any $z \in X(\overline{K}_{v_0})$
 and any $\zeta \in X(\overline{K}_{v_0})  - E_{v_0}$ is the infimum of $G(z,\zeta;E'_{v_0})$
 over compact subsets $E'_{v_0}$ of $E_{v_0}$.  Furthermore, we have $G(z,\zeta;E_{v_0}) < G(z,\zeta;E'_{v_0})$ for $z \in X(\overline{K}_{v_0})  - E_{v_0}$ by
 the computations in \cite[\S 5.2.B]{Rumely1} since we took $E_{v_0}$ to be $\mathcal{X}$-trivial.
 So we can take $E'_{v_0}$ to be a compact subset of $E_{v_0}$ such that
 the global Green's matrix $\Gamma(\mathcal{X},\mathcal{E})$ defined in \cite[Theorem 5.1.4]{Rumely1} differs from $\Gamma(\mathcal{X},\mathcal{E}')$ by a matrix with positive entries that are arbitrarily close to $0$.  Then $E'_{v_0}$
 is capacitifiable by \cite[Theorem 4.3.4]{Rumely1}, so (1) holds.  The value
 of the game defined by $\Gamma(\mathcal{X},\mathcal{E}')$ is  larger than
 that defined by $\Gamma(\mathcal{X},\mathcal{E})$, so we get (2);  see \cite[p. 327-328]{Rumely1}. The log of
 the Green's metric on $L$ at $v_0$ associated with $E'_{v_0}$ and with $E_{v_0}$ differs
 by a constant we can make arbitrarily close to $0$, so we get (3) from (\ref{eq:metric}).
To prove  (4), we first note that hypothesis 
  (\ref{eq:capone}) in Theorem \ref{thm:caponethm} implies the following.  When we write $D = \sum_{x \in \mathcal{X}} n_x x$,
  then $n_x > 0$ for all $x$ and the probability vector $P = (n_x/\mathrm{deg}(D))_{x \in \mathcal{X}}$
  must define  an optimum strategy for the game associated to $\Gamma(\mathcal{X},\mathcal{E})$.  
  Furthermore $S_\gamma(\mathcal{E},D) = 0$ says that this optimum strategy achieves
  value $0$. Since $\Gamma(\mathcal{X},\mathcal{E}') - \Gamma(\mathcal{X},\mathcal{E})$
  has all positive entries, playing $P$ in the game defined by $\Gamma(\mathcal{X},\mathcal{E}') $
  leads to a  positive value.  This means that the construction of Rumely in \cite[\S 6, Corollary 6.2.7]{Rumely1}
  produces a function with the properties in (4).    
\end{proof}

\begin{lemma} 
\label{lem:reduce} Let $\mathcal{E}'$ be as in Lemma \ref{lem:shrink}.  There is a constant $c$
independent of $n$ such $H^0(X,L^{\otimes n})$ has a basis of sections $s$ for which $\lambda'(s) \ge c$.
\end{lemma}

\begin{proof} Let $f$ in part (4) of Lemma \ref{lem:shrink} have divisor $mD$ for some $m > 0$.
By Riemann-Roch, we can find a finite subset $\{h_j\}_{j \in J}$ of elements of the function
field $K(X)$ with the following properties. The poles of the $h_j$ are supported on $\mathcal{X} = \mathrm{supp}(D)$,
and the height $\lambda(h_j)$ of $h_j$ with respect to the Green's metrics associated to 
$\mathcal{E}'$ is finite.  Further, for all $n$, the collection of functions $\{h_j f^i\}_{j \in J, 0 \le i }$ contains a basis for $H^0(X,L^{\otimes n}) = H^0(X,\mathcal{O}_X(nD))$.   Now Lemma \ref{lem:greensup}  gives 
$$\lambda'(h_j f^i) = - \sum_v k_v \ln(\mathrm{sup}(h_j f^i,E'_v)) \ge - \sum_v k_v \ln( \mathrm{sup}(h_j,E'_v)) = \lambda'(h_j)$$
because $\mathrm{sup}(f^i,E'_v) \le 1$ for all $i$ by Lemma \ref{lem:shrink}.  Since there
are finitely many $h_j$, this proves the Lemma.
\end{proof}

\begin{remark}
Lemma \ref{lem:reduce} could be deduced from a result of Zhang in \cite[Thm. 4.2]{Zhang1} about arithmetic ampleness by verifying that
the capacity theoretic metrics involved satisfy the hypotheses of this result.  
\end{remark}

\medbreak
\noindent {\bf Proof of Theorem \ref{thm:caponethm}}
\medbreak
Let $c$ be as in Lemma \ref{lem:reduce}. In view of Lemmas \ref{lem:reduce}
and \ref{lem:shrink}, for each $\epsilon > 0$, there is a basis of sections $s$ of $H^0(X,L^{\otimes n})$
such that when $\lambda(s)$ is the height function associated to 
the Green's metrics coming from $\mathcal{E}$, we have $\lambda(s)/n \ge c/n - \epsilon$.
Letting $n \to \infty$ shows that the limiting measure $\nu$ associated to the
ratios $\lambda_i/n$ as $\lambda_i$ ranges over the successive maxima
of $H^0(X,L^{\otimes n})$ can have no support on the negative real axis.  On
the other hand,  Lemma \ref{lem:no small} shows the support is also trivial
on the positive real axis.  So $\nu$ must be the Dirac measure supported at $0$.
\medbreak

\subsection{Measures associated to zeros of small sections}
\label{s:smallzeros}

The object of this subsection is to prove Theorems \ref{thm:succmin} and \ref{thm:succminberk}.
Accordingly we suppose  
there is a morphism $h:X \to \mathbb{P}^1$ such that $D= h^*(\infty) = N \zeta$ for some point $\zeta \in X(K)$,
where $N = \mathrm{deg}(h)$.  We also suppose $$E_v = \{z \in X(\overline{K}_v): 
|h(z)|_v \le 1\} = h^{-1}(B_v)$$ for all $v$ when $$B_v = \{z \in \overline{K}_v: |z|_v \le 1\} \subset \mathbb{P}^1(\overline{K}_v) - \{\infty\}.$$  
  Here $h^* \mathcal{O}_{\mathbb{P}^1}(\infty) = \mathcal{O}_X(D)$
and for all $n > 0$ we have a global section $\Phi_n(z)$ in $H^0(\mathbb{P}^1, \mathcal{O}_{\mathbb{P}^1}(n \infty))$ for any integer $a$ when $\Phi_n(z)$ is the $n^{\mathrm{th}}$ cyclotomic polynomial.  The set $\mathrm{zer}(\Phi_n(z))$ of zeros of $\Phi_n(z)$ is just the set of all primitive $n^{\mathrm{th}}$
roots of unity.  

We now fix a place $v$ of $K$.  Define $\overline{B}_v = B_v$ and $\overline{E}_v = E_v$ if $v$ is archimedean. If $v$ 
is non-archimedean, we  let $\overline{B}_v$
be the closure of $B_v$ in $\mathbb{P}^1_{Berk}$ and we let $\overline{E}_v$ be the closure of $E_v$
in $X_{Berk,\mathbb{C}_v}$.  

Suppose first that $v$ is archimedean.  In Theorem \ref{thm:succmin} 
we let $\mu_0$ be the uniform measure on the boundary of the unit
disk $B_v = \{z \in \mathbb{P}^1(\overline{K}_v) = \mathbb{P}^1(\mathbb{C}): |z|_v = 1\}$,
and we defined $\mu(E_v,D)$ to be $\frac{1}{N} h^{-1}(\mu_0)$, where 
$D = N\zeta$ is the polar divisor of $h$.  By \cite[Prop. 4.1.25]{Rumely1}, $\mu(E_v,D)$
is the equilibrium measure of $E_v$ with respect to $D$.

Suppose now that  $v$ is non-archimedean. The probability measure $\mu(E_v,D)$ 
on $X_{Berk,\mathbb{C}_v}$ described in Theorem  \ref{thm:succminberk} is well defined by \cite[\S 2.3]{CL1} and  the paragraph following \cite[Theorem 3.1]{CL1}.

We claim that for all $v$, 
\begin{equation}
\label{eq:limanswer}
\mu(\overline{E}_v,D) = \lim_{m \to \infty} \delta_{\mathrm{zer}(h^*(\Phi_{2^m}(z))}
\end{equation}
where $h^*(\Phi_{2^m}(z))$ is a section of $H^0(X,h^* \mathcal{O}_{\mathbb{P}^1}(2^m\ \infty)) = H^0(X,\mathcal{O}_X(nD))$
and $\delta_{\mathrm{zer}(h^*(\Phi_{2^m}(z))}$ is the Dirac measure associated to the zeros of this section.

Suppose first that $v$ is archimedean and that $h:X \to \mathbb{P}^1$ is the identity map.  The zeros of $h^*(\Phi_{2^m}(z))$ are simply all the odd powers of a primitive root of unity of order $2^m$.  Then (\ref{eq:limanswer}) is clear
from the fact that in this case, $\overline{E}_v = \overline{B}_v$ is the unit disc about the origin, so
$\mu(\overline{E}_v,D)$ is the uniform measure on the boundary of the unit disc.  For archimedean
$v$, the case of all $h:X \to \mathbb{P}^1$ satisfying our hypotheses follows from this and the the fact that $\mu(E_v,D) = \frac{1}{N} h^{-1}(\mu_0)$.
 
 Suppose now that $v$ is non-archimedean.  As in Theorem \ref{thm:succminberk} let ${\bf h}:\mathcal{X} \to \mathbb{P}^1_{\mathcal{O}_K}$ be the minimal regular model of the finite morphism $h:X \to \mathbb{P}^1_K$. 
 We give the line bundle $\mathcal{O}_{\mathbb{P}^1_{\mathcal{O}_K}}(\overline{\infty}) = \mathcal{L}$ on $\mathbb{P}^1_{\mathcal{O}_K}$ the adelic metric associated to the Weil height.  Then $\mathbb{P}^1_{\mathcal{O}_K}$ and the divisors defined by the zeros of $\Phi_{2^m}(z)$ have height equal to $0$.  We give ${\bf h}^*\mathcal{L}$ the pull back of the adelic metric of $\mathcal{L}$.  For any cycle $Z$ on $\mathcal{X}$ we have from
 \cite[Prop. 3.2.1]{BostGS} that
 \begin{equation}
 \label{eq:BGSrule}
 H_{{\bf h}^*\mathcal{L}}(Z) = H_{\mathcal{L}}({\bf h}_*Z)
 \end{equation}
 where $H_{\mathcal{L}}$ here is the height before normalization that is defined in \cite[\S 3.1.1]{BostGS}.  If 
 $Z$ is the cycle $\mathcal{X} = {\bf h}^* \mathbb{P}^1$ we have ${\bf h}_* {\bf h}^*\mathbb{P}^1 = N \cdot \mathbb{P}^1$ by the projection formula so we conclude $H_{{\bf h}^*\mathcal{L}}(X) = 0$  Suppose now that
 $Z$ is a cycle contained in the divisor of zeros of $h^*(\Phi_{2^m}(z))$. Then ${\bf h}_*Z $ is contained in the divisor of zeros of $\Phi_{2^m}(z)$, and so $H_{\mathcal{L}}({\bf h}_* Z)) = 0$.  Thus $H_{{\bf h}^*\mathcal{L}}(Z) = 0$
 by (\ref{eq:BGSrule}).  By \cite[\S 3.1.4]{BostGS}, the same is now true if we replace $Z$ by 
 by any cycle contained in the base change of $Z$ by a morphism $\mathrm{Spec}(\mathcal{O}_{K'}) \to \mathrm{Spec}(\mathcal{O}_K)$ associated to a finite extension $K'$ of $K$. 
 We conclude that the Galois conjugates of any zero of $h^*(\Phi_{2^m}(z))$ have adelic height $0$ with respect
 to the above adelic metric on ${\bf h}^*(\mathcal{L})$, and this is also the height of $\mathcal{X}$ with respect to 
 this metric.  So these zeros as $m \to \infty$ form a generic sequence of points of $X(\overline{K}_v)$ 
 in the sense of \cite[Thm. 3.1]{CL1}.  Now \cite[Thm. 3.1]{CL1} shows that the limit on 
 the right hand side of (\ref{eq:limanswer}) equals the Berkovich measure described in
 just before \cite[Example 3.2]{CL1}, and this equals the measure $\mu(\overline{E}_v,D)$
 defined in Theorem  \ref{thm:succminberk}.  We have now shown (\ref{eq:limanswer}) in all cases.
 
Consider the normalized height $\lambda(h^*(\Phi_{2^m}(z)))$ of $h^*(\Phi_{2^m}(z))$  with respect to the Green's metrics on $\mathcal{O}_X(2^mD) = h^* \mathcal{O}_{\mathbb{P}^1}(2^m \infty)$ associated to $\mathcal{E} = \prod_v E_v$.  We have $\lambda(h^*(\Phi_{2^m}(z))) \to 0$ as $m \to \infty$ because  $\Phi_{2^m}(z) = (z^{2^m} - 1)/(z^{2^{m-1}} - 1)$ has normalized height tending toward $0$ with respect to the Green's metrics on $\mathcal{O}_{P^1}(\infty)$ which are associated to $\mathcal{B} = \prod_w B_w$.   Since the zero sets of the $\Phi_{2^m}(z)$ are disjoint for different $m$, 
for sufficiently large $m$ the zeros of   $h^*(\Phi_{2^m}(z))$ will avoid any prescribed finite subset of $X(\overline{K}_v)$.
Hence the limit measure $\mu(\overline{E}_v,D)$ in (\ref{eq:limanswer}) has mass $0$ at every point of $X(\overline{K}_v)$, so
Theorem \ref{thm:Serrethm} shows $\mu(\overline{E}_v,D)$  lies in $\cap_{\lambda < 0} M_v(T(L,\lambda),\infty)$.  
This completes the proof of 
Theorems \ref{thm:succmin} and \ref{thm:succminberk}.


\begin{thebibliography}{88}

\bibitem{AbbesBouche}  A. Abbes and T. Bouche, ``Th\'eor\`eme de Hilbert-Samuel `arithm\'etique'," Ann. Inst. Fourier (Grenoble) 45 (1995), no. 2, 375 -- 401. 

\bibitem{A}Arakelov, S.Ju., 
Intersection theory of divisors on an arithmetic surface. (Russian, English)
Math. USSR, Izv. 8(1974), 1167-1180 (1976), translation from Izv. Akad. Nauk SSSR, Ser. Mat. 38 (1974), 1179-1192.


\bibitem{Berk90}V. G. Berkovich, Spectral theory and analytic geometry over non-Archimedean fields, Mathematical Surveys and Monographs, vol. 33, American Mathematical Society, Providence, RI, 1990.


\bibitem{BV} J-M. Bismut and E.Vasserot, The asymptotics of the Ray-Singer analytic torsion associated with high powers of a positive line bundle,
Communications in Math. Physics  125  (1989), 355-367.

\bibitem{BostGS} J.-B. Bost, H. Gillet and C. Soul\'e, Heights of projective varieties and positive green forms, Journal of the American Mathematical Society, Vol. 7, no. 4 (1994), 903 - 1027.

\bibitem{BostLetter} J.-B. Bost, Potential theory and Lefschetz theorems for arithmetic surfaces. (English, French summary) 
Ann. Sci. \'Ecole Norm. Sup. (4) 32 (1999), no. 2, 241 - 312. 


\bibitem{BOU} T. Bouche, Convergence de la m\'etrique de Fubini-Study d'un fibr\'e lin\'eaire positif. (Convergence of the Fubini-Study metric of a positive line bundle.). (French) 
Ann. Inst. Fourier 40 (1990), 117-130.


 
\bibitem{CL1} A. Chambert-Loir, Mesures et  \'equidistribution sur des espaces de Berkovich,
J. reine angew. Math. 595 (2006), ,215-235, math.NT/0304023.




\bibitem{Chen} H. Chen, ``Convergence des polygones de Harder-Narasimhan," M\'emoires de la Soci\'et\'e Math\'ematique de France 120 (2010), 1-120. 
 
 \bibitem{Chinburg1} T. Chinburg, ``Capacity theory on varieties,"
Compositio Math., 80 no. 1, 77 - 84, 1991.

\bibitem{CMPT} T. Chinburg, G. Pappas, L. Moret-Bailly and M. J. Taylor, ``Finite morphisms to projective space and capacity theory," 
J. Reine Angew. Math. 727 (2017), 69 - 84. 


\bibitem{ChinburgRumely}T. Chinburg and R. Rumely, ``Well-adjusted models for curves over Dedekind rings,"
in Arithmetic algebraic geometry (Texel, 1989), Progr. Math. 89, Birkh\"auser Boston, Boston, MA. (1991).

\bibitem{D}Deligne, P., 
Le d\'eterminant de la cohomologie. Current trends in arithmetical algebraic geometry (Arcata,
Calif., 1985), ,
Contemp. Math., 67, Amer. Math. Soc., Providence, RI, (1987), 93-177.


\bibitem{Gaudron} E. Gaudron, ``Pentes des fibr\'es vectoriels ad\'eliques sur un corps global'',  Rendiconti del Semi-
nario Matematico della Universita di Padova, 2008, 119, 21-95.

\bibitem{ARR}
H. Gillet and C. Soul{\'e}, An arithmetic {R}iemann-{R}och
  theorem, Invent. Math. 110  (1992),  473--543.

\bibitem{GS-Comptes} H. Gillet and C. Soul\'e, ``Amplitude arithmŽtique," C. R. Acad. Sci. Paris S\'er. I Math. 307 no. 17 (1988), 887--890.



\bibitem{HoloSound} R. Holowinsky and K. Soundararajan, ``Mass equidistribution for Hecke eigenforms,"
Annals of Mathematics, Second Series, Vol. 172, No. 2 (September, 2010), pp. 1517- 1528.

\bibitem{Kuhn} U. K\"uhn, Generalized arithmetic intersection numbers, 
J. reine angew. Math. 534 (2001), 209 - 236.


 \bibitem{RLV} C. F. Lau, R. Rumely, R. Varley, Existence of the sectional capacity,
Mem. Amer. Math. Soc., {\bf 145},
2000.


 \bibitem{Ny} D. Witt-Nystr$\hat{\mathrm{o}}$m, Transforming metrics on a line bundle to the Okounkov body, Ann. Sci. \'Ecole Norm. Sup. (4) 47, No. 6,  (2014), pp. 1111-1161.



\bibitem{Rudnick} Z. Rudnick,  On the asymptotic distribution of zeros of modular forms, Int. Math. Res. Notices (2005), pp. 2059-2074. 

 \bibitem{Rumely1} R. Rumely,  Capacity Theory on Algebraic Curves, Springer Lectures Notes 1378,
 Springer Verlag (1989).
 
 \bibitem{Rumely2} R. Rumely,  Capacity theory with local rationality. The strong Fekete-Szegš theorem on curves. 
 Mathematical Surveys and Monographs, 193. American Mathematical Society, Providence, RI. (2013). 
 
\bibitem{RumelyDuke} R. Rumely,  On the relation between Cantor's capacity and the sectional capacity, Duke Math. J. 70 (1993),  517-574.
 

\bibitem{Serre} J. P. Serre, Distribution asymptotique des valeurs propres des endomorphismes de Frobenius [d'apr\`es Abel, Chebyshev, Robinson, ...] 
S\'em. Bourbaki, (2017-2018), n. 1146, p. 1 - 32.



\bibitem{Shimura} G. Shimura, The Special Values of the Zeta Functions Associated with Cusp Forms, Comm. Pure and Applied Math., vol. 29 (1976), p. 783 - 804.

\bibitem{MI}C. Soul\'e,  Successive minima on arithmetic varieties, Compositio Math.   96  (1995),  85-98.



\bibitem{TI} G. Tian,  K\"ahler metrics on algebraic manifolds, Ph. D. Thesis, Harvard University (1988), 57 p.

\bibitem{YuanBig} X. Yuan, Big line bundles over arithmetic varieties, Invent. Math. 173 (2008), 603-649.

\bibitem{Yuan} X. Yuan,  On Volumes of Arithmetic Line Bundles II, Preprint available at https://arxiv.org/abs/0909.3680.

\bibitem{Zhang1} S. Zhang, Positive line bundles on arithmetic varieties, J. A. M. S. 8 no. 1 (1995), 187 - 221.

\bibitem{Zhang} S. Zhang, Small  points  and  adelic  metrics, J.  Algebraic  Geom.  4  (1995),  no. 2.

 
\bibliography{FRGBIB}

\end{thebibliography}
\end{document}